\documentclass[a4paper,reqno,oneside,12pt]{amsart}
\usepackage{fullpage}

\usepackage{amsmath}
\usepackage{amsfonts}
\usepackage{amssymb}
\usepackage{verbatim,enumitem,graphics}
\usepackage[colorlinks]{hyperref}
\usepackage{dsfont}

\usepackage[utf8]{inputenc}

\usepackage{enumitem}
\usepackage{MnSymbol}

\setenumerate{leftmargin=*} 

\newtheorem{theorem}{Theorem}[section]
\newtheorem{lemma}[theorem]{Lemma}

\newtheorem{conjecture}[theorem]{Conjecture}
\newtheorem{definition}[theorem]{Definition}

\newtheorem{claim}[theorem]{Claim}

\newcommand{\oldqed}{}
\def\endofClaim{\hfill\scalebox{.6}{$\Box$}}

\newenvironment{claimproof}[1][Proof]{
  \renewcommand{\oldqed}{\qedsymbol}
  \renewcommand{\qedsymbol}{\endofClaim}
  \begin{proof}[#1]
}{
  \end{proof}
  \renewcommand{\qedsymbol}{\oldqed}
}

\newcommand{\By}[2]{\overset{\mbox{\tiny{#1}}}{#2}}
\newcommand{\ByRef}[2]{   \By{\eqref{#1}}{#2} }

\newcommand{\eqByRef}[1]{ \ByRef{#1}{=} }

\newcommand{\geByRef}[1]{ \ByRef{#1}{\ge} }
\newcommand{\NN}{\mathbb{N}}
\newcommand{\EMAIL}[1]{  \textit{E-mail}: \texttt{#1} }
\newcommand{\Gnp}{G(n,p)}

\newcommand{\cH}{\mathcal{H}}

\newcommand{\cF}{\mathcal{F}}
\newcommand{\cE}{\mathcal{E}}

\newcommand{\cP}{\mathcal{P}}
\newcommand{\cJ}{\mathcal{J}}
\newcommand{\cS}{\mathcal{S}}

\newcommand{\cI}{\mathcal{I}}

\newcommand{\PP}{\mathbb{P}}
\newcommand{\EE}{\mathbb{E}}

\newcommand{\eps}{\varepsilon}

\renewcommand{\subset}{\subseteq}

\def\itm#1{\rm ({#1})}
\def\itmit#1{\itm{\it #1\,}}
\def\rom{\itmit{\roman{*}}}

\def\romprime{\itmit{\roman{*}'}}

\title{Embedding spanning bounded degree graphs in randomly perturbed graphs}

\author[J. Böttcher]{Julia Böttcher*}
\thanks{*
    Department of Mathematics, London School of Economics, Houghton Street,
London WC2A 2AE, UK.
   \EMAIL{j.boettcher@lse.ac.uk}
 }

\author[R. Montgomery]{Richard Montgomery\dag}
\thanks{
    \dag
University of Birmingham, Birmingham, B15 2TT, UK.
   \EMAIL{r.h.montgomery@bham.ac.uk}
 }

\author[O. Parczyk]{Olaf Parczyk \ddag}
\author[Y. Person]{Yury Person \ddag}

\thanks{\ddag Institut f\"ur Mathematik, Technische Universit\"at Ilmenau, 98684 Ilmenau, Germany.
	\EMAIL{\{olaf.parczyk|yury.person\}@tu-ilmenau.de}
}
\thanks{The research leading to this paper was initiated during the
  workshop on
  `Large-Scale Structures in Random Graphs' at the Alan Turing Institute,
  which was financially
  supported by the Heilbronn Institute for Mathematical Research, the Alan
  Turing Institute, and the
  Department of Mathematics at LSE.
  JB is partially supported by EPSRC (EP/R00532X/1).
  OP and YP were supported by DFG grant PE 2299/1-1.}

\begin{document}

\begin{abstract}
We study the model $G_\alpha\cup G(n,p)$ of randomly perturbed dense
graphs, where~$G_\alpha$ is any $n$-vertex graph with minimum degree at
least $\alpha n$ and $G(n,p)$ is the binomial random graph.
We introduce a general approach for studying the appearance of spanning
subgraphs in this model using absorption. This approach yields simpler proofs of several
known results. We also use it to derive the following two new results.

For every $\alpha>0$ and $\Delta\ge 5$, and every $n$-vertex graph $F$ with
maximum degree at most $\Delta$, we show that if
$p=\omega(n^{-2/(\Delta+1)})$ then $G_\alpha \cup G(n,p)$ with high probability
contains a copy of $F$. The bound used for $p$ here is lower by a
$\log$-factor in comparison to the conjectured threshold for the
general appearance of such subgraphs in $G(n,p)$ alone, a typical feature
of previous results concerning randomly perturbed dense graphs.

We also give the first example of graphs where the appearance threshold in
$G_\alpha \cup G(n,p)$ is lower than the appearance threshold in $G(n,p)$
by substantially more than a $\log$-factor.
We prove that, for every $k\geq 2$ and $\alpha >0$, there is some
$\eta>0$ for which the $k$th power of a Hamilton cycle with high probability
appears in $G_\alpha \cup G(n,p)$ when $p=\omega(n^{-1/k-\eta})$. The
appearance threshold of the $k$th power of a Hamilton cycle in $G(n,p)$
alone is known to be $n^{-1/k}$, up to a $\log$-term when $k=2$, and
exactly for $k>2$.
\end{abstract}

\maketitle

\section{Introduction and Results}

Many important results in Extremal Graph Theory and in Random Graph Theory
concern the appearance of spanning subgraphs in dense graphs and in random
graphs, respectively.  In Extremal Graph Theory, minimum degree conditions
forcing the appearance of such subgraphs are studied.  For example, Dirac's
Theorem~\cite{dirac1952some}, one of the cornerstones of Extremal Graph
Theory, states that an $n$-vertex graph with minimum degree at least $n/2$
has a Hamilton cycle when $n\geq 3$. In Random Graph Theory, on the other hand, bounds are
sought on the probability threshold for the appearance of subgraphs in a
random graph.
Let $\Gnp$ be the \emph{binomial random graph} model with vertex set $[n]$,
where each possible edge is chosen independently at random with probability~$p$.
We say that $\Gnp$ has some property $\cP$ \emph{with high probability} (whp)
if $\lim_{n\to\infty}\PP [\Gnp \in \cP]=1$.
A key result by P\'{o}sa~\cite{Pos76} and
Kor\v{s}unov~\cite{Kor76} is that $G(n,p)$ with high probability contains a
Hamilton cycle if $p=\omega(\log n/n)$, whereas if $p=o(\log n/n)$ then
$G(n,p)$ with high probability does not. Here, we write $p(n)=\omega\big(f(n)\big)$
to signify $p(n)/f(n)\to\infty$, and $p(n)=o\big(f(n)\big)$
to signify $p(n)/f(n)\to 0$.

The study of randomly perturbed graphs combines these two approaches by
taking the union of a graph satisfying some minimum degree condition and a
random graph $G(n,p)$. The goal is then to determine which minimum degree conditions and
edge probabilities suffice to guarantee some given subgraph with high probability.
Bohman, Frieze and Martin~\cite{bohman2004adding}, who pioneered the study
of randomly perturbed graphs, proved that for every $\alpha>0$ the union of
every $n$-vertex graph with minimum degree at least $\alpha n$ and a random
graph $G(n,p)$ with $p=\omega(1/n)$ contains whp a Hamilton
cycle. This result shows that, compared to Dirac's Theorem, a much smaller
minimum degree condition suffices in a randomly perturbed graph, and
compared to the random graph $G(n,p)$ alone a $\log$-term improvement in the
edge probability is possible.

The recent increased interest in randomly perturbed graphs sparked a
collection of results of a similar flavour, typically featuring a small
linear minimum degree condition and a $\log$-term improvement in the edge
probability. In this paper, we contribute to this body of research by
developing a new general method for establishing such results for spanning
subgraphs. Our approach uses an absorbing method. We show that this
new approach gives simpler proofs of a number of known results, whose
original proofs often
use the regularity method and are therefore technically more complex.
It
also allows us to give strong new results
concerning powers of Hamilton cycles and
general bounded degree spanning subgraphs in randomly perturbed
graphs. In particular, our result on powers of Hamilton cycles
provides the first example for graphs with
an $n^{\Omega(1)}$ improvement in the edge probability compared to $G(n,p)$.
A similar phenomenon was already discovered in the context of hypergraphs
by McDowell and Mycroft~\cite{MM_HyperPeturbed}, which we will return to in our concluding remarks.

Before discussing our techniques and results in more detail, we set our work in context by summarising related results in random graphs and randomly perturbed graphs.

\subsection{Thresholds in \texorpdfstring{$\Gnp$}{G(n,p)}}

We say that the function $\hat{p} : \mathbb{N} \rightarrow [0,1]$ is a
\emph{threshold} for a graph property $\cP$, if
\[\lim_{n\to\infty}\PP [\Gnp \in \cP] = \begin{cases}
    0 \qquad &\text{whenever $p = o( \hat{p})$}, \text{ and} \\
    1 & \text{whenever $p = \omega (\hat{p})$}\,.
\end{cases}
\]
If only the latter is known to be true, then we say that
$\hat{p}$ is an upper bound for the threshold for $\cP$ in $\Gnp$.
Containing a graph as a (not necessarily induced) subgraph is a monotone
property and therefore it has a threshold by a result of Bollob\'as and
Thomason~\cite{bollobas1987threshold}.
In the following, we will focus on spanning subgraphs.

In their seminal work, Erd\H{o}s and R\'{e}nyi~\cite{ErdRen66} proved that the threshold for perfect matchings in $\Gnp$ is $\log n/n$.
P\'{o}sa~\cite{Pos76} and Kor\v{s}unov~\cite{Kor76} independently showed that the property of having a Hamilton cycle has the same threshold.

The problem of finding powers of Hamilton cycles as a subgraph is generally considered
a stepping stone towards results for more general spanning subgraphs.
The \emph{$k$th power} $G^{(k)}$ of a graph $G$ is the graph obtained from $G$ by connecting all vertices at distance at most $k$.
K\"uhn and Osthus~\cite{KO12} observed that the threshold in $G(n,p)$ for the $k$th power of a Hamilton cycle when
$k\ge 3$ is $n^{-1/k}$; this follows
from a general embedding theorem due to Riordan~\cite{Riordan} (see Theorem~\ref{theorem:Riordan}). Similarly, the threshold of the
square of a Hamilton cycle is conjectured to be $n^{-1/2}$, but this is
still open. Currently, the best known upper bound, by Nenadov and
\v{S}kori\'c~\cite{nenadov2016powers}, is off by a $O(\log^4n)$-factor from this conjectured threshold.

For a graph~$H$, an \emph{$H$-factor} on~$n$ vertices is the vertex disjoint
union of copies of~$H$ with~$n$ vertices in total. An \emph{almost  $H$-factor} in an $n$-vertex graph~$G$ is a subgraph of~$G$ that is an
$H$-factor on~$(1-\varepsilon)n$ vertices.
A breakthrough result was achieved by
Johansson, Kahn and Vu~\cite{JohanssonKahnVu_FactorsInRandomGraphs} who
showed that the threshold for a $K_{\Delta+1}$-factor, that is
$\frac{n}{\Delta+1}$ vertex-disjoint copies of $K_{\Delta+1}$, is given by
\begin{align*} p_\Delta= \Big(\frac{\log^{1/\Delta} n}{n}\Big)^{2/(\Delta+1)}.
\end{align*}
In fact, their result concerns, more generally, $H$-factors for
strictly balanced graphs~$H$.
The \emph{$1$-density} of a graph~$H$ on at least~$2$ vertices is
\[m_1(H)=\max_{H'\subseteq H, v(H')>1}\frac{e(H')}{v(H')-1}\,,\] and a graph is called
\emph{strictly balanced} if~$H$ is the only maximiser in $m_1(H)$.
Johansson, Kahn and Vu~\cite{JohanssonKahnVu_FactorsInRandomGraphs} proved that for factors of strictly balanced graphs $H$, the threshold
is $n^{-1/m_1(H)} \log^{1/e(H)}n$.
Gerke and McDowell~\cite{GerMcD}, on the other hand, showed that for
certain (but not all) graphs~$H$ which are not strictly balanced, this threshold is
$n^{-1/m_1(H)}$.

Let us now turn to larger classes of graphs.
For bounded degree spanning trees, the second author~\cite{MontTrees} showed
that, for each fixed $\Delta$, $\log n/n$ is  the appearance threshold for
single spanning trees with maximum degree at most $\Delta$ (see also~\cite{M14a}).

More generally,
let $\cF(n,\Delta)$ be the family of graphs on $n$ vertices with maximum degree at most $\Delta$.
For some constant $C$, Alon and Füredi~\cite{AlonFuredi_SpanningSubgraphs} proved that, if $p \ge C (\log n/n)^{1/\Delta}$,  then $\Gnp$ contains any single graph from $\cF(n,\Delta)$ whp.
This is far from optimal and, since the clique-factor is widely believed to have the highest appearance threshold among the graphs in $\cF(n,\Delta)$, the following well-known conjecture is natural.
\begin{conjecture}\label{thresconj}
If $\Delta\in \NN$, $F \in \cF(n,\Delta)$ and $p = \omega(p_\Delta)$, then $\Gnp$  whp contains a copy of~$F$.
\end{conjecture}
For $\Delta=2$, this conjecture was very recently resolved by Ferber, Kronenberg and Luh~\cite{ferber2016optimal}, who in fact showed a stronger \emph{universality} statement,
where all graphs in $\cF(n,\Delta)$ are found simultaneously.
For larger $\Delta$, Riordan~\cite{Riordan} gave a general result (see Theorem~\ref{theorem:Riordan}), which requires an edge probability within a factor of $n^{\Theta(1/\Delta^2)}$ from $p_\Delta$.
The current best result in the direction of Conjecture~\ref{thresconj} is the following \emph{almost} spanning version by Ferber, Luh and Nguyen~\cite{ferber2016embedding}.
\begin{theorem}[Ferber, Luh and Nguyen~\cite{ferber2016embedding}]
\label{Theorem:Ferber}
Let $\varepsilon > 0 $ and $\Delta \ge 5 $.
For every $F \in \cF((1-\varepsilon)n, \Delta)$ and $p =  \omega
(p_\Delta)$ the random graph $\Gnp$ whp contains a copy of $F$.
\end{theorem}
The approach in~\cite{ferber2016embedding} is based on ideas from Conlon, Ferber, Nenadov and \v{S}kori\'{c}~\cite{conlon2016almost}, who proved a stronger universality statement for the almost spanning case while using the edge probability $n^{-1/(\Delta-1)} \log^5 n$. Theorem~\ref{Theorem:Ferber} for $\Delta=3$ was thus already known (up to a $\log$-factor), whereas the case for $\Delta=4$ remains open.
For spanning subgraphs, very recently, Ferber and Nenadov~\cite{FerbeNenadovSpanning} showed that
for $p \ge (\log^3 n/n)^{1/(\Delta-1/2)}$ the random graph $\Gnp$ whp
contains all graphs in $\cF(n,\Delta)$ universally.

In the almost spanning case, the $\log$-term in $p_\Delta$ is expected to be redundant~\cite{ferber2016embedding}, but this remains open. In this paper, we will show that the $\log$-term in $p_\Delta$ is redundant, even in the spanning case, if we add $\Gnp$ to a deterministic graph with linear minimum degree.

\subsection{Randomly perturbed graphs}

Bohman, Frieze and Martin~\cite{bohman2003many} introduced the following model of randomly perturbed graphs.
For $\alpha \in (0,1)$ and an integer~$n$, we first let $G_\alpha$ be any
$n$-vertex graph with minimum degree at least $\alpha n$.
We then reveal more edges among the vertices of~$G_\alpha$ independently at
random with probability~$p$.
The resulting graph $G_\alpha \cup \Gnp$ is a \emph{randomly perturbed
  graph} and we are interested in its properties. In particular,
research has focused on
comparing thresholds in $G_\alpha \cup \Gnp$ to thresholds in $\Gnp$.

Again, we concentrate on spanning subgraphs. Note that the existence of such
subgraphs in $G_\alpha \cup \Gnp$ is a monotone property (in $\Gnp$),
and thus has a threshold.
Of course, if $\alpha \ge 1/2$, then $G_\alpha$ is itself Hamiltonian by Dirac's Theorem. 
For $\alpha \in (0,1/2)$, Bohman, Frieze and Martin~\cite{bohman2003many} showed the existence of some $c=c(\alpha)>0$ so
that, if $p = c/n$, then, for any $G_\alpha$, there is a Hamilton cycle in $G_\alpha \cup \Gnp$ whp.
They also proved that this is optimal: there exists some $c'>0$ so that, there are graphs $G_\alpha$ such that $G_\alpha \cup G(n,c'/n)$ is not Hamiltonian whp.
Comparing this threshold to the threshold for Hamiltonicity in $\Gnp$ we
note an extra factor of $\log n$ in the latter. This $\log n$ term is
necessary to guarantee minimum degree at least $2$ in $\Gnp$ -- otherwise
clearly no Hamilton cycle exists. In the model $G_\alpha\cup \Gnp$, however, this already holds in $G_\alpha$ alone.

Krivelevich, Kwan and Sudakov~\cite{krivelevich2015bounded} studied the corresponding problem for the containment of spanning  trees of maximum degree $\Delta$ in $G_\alpha\cup \Gnp$.
For  $p= c(\eps,\Delta)/n$ it is already possible to find any \emph{almost}
spanning bounded degree tree on $(1-\eps)n$ vertices in $\Gnp$~\cite{alon2007embedding}.
The addition of $G_\alpha$ then ensures there are no isolated vertices, and Krivelevich, Kwan and Sudakov~\cite{krivelevich2015bounded} showed that this indeed allows every vertex to be incorporated into the embedding.
They thus prove that, for $\alpha>0$, maximum degree $\Delta$ and $p= c(\alpha,\Delta)/n$ every spanning bounded degree tree is contained in $G_\alpha\cup \Gnp$.

Very recently, Balogh, Treglown and Wagner~\cite{BTW_tilings} determined the threshold of appearance for general factors in the model $G_\alpha \cup \Gnp$.
They proved that for every~$H$, if $p=\omega(n^{-1/m_1(H)})$, then $G_\alpha\cup \Gnp$ contains an $H$-factor whp.
Comparing this to the result of Johansson, Kahn and Vu~\cite{JohanssonKahnVu_FactorsInRandomGraphs}, we observe again a
saving of a $\log$-term. For the graphs~$H$ covered by the result of Gerke
and McDowell~\cite{GerMcD}, on the other hand, we see that the thresholds in
$G_\alpha\cup \Gnp$ and in $\Gnp$ are the same.

Other monotone properties considered in the randomly perturbed graph model
include containing a fixed sized clique, having small diameter,
being $k$-connected~\cite{bohman2004adding}, and being non-$2$-colourable~\cite{sudakov2007many}.

\subsection{Our Results}
Our main contribution to the study of randomly perturbed graphs
is the introduction of a new approach for obtaining results concerning
spanning subgraphs.
The basic idea is to use some random edges with the assistance of the deterministic edges to create so-called \emph{reservoir sets}.
Our key technical result is Theorem~\ref{newtheorem}, which gives a
condition for applying this method to spanning subgraphs. We defer the
statement of this result along with the necessary definitions to Section~\ref{sec:aa}.

Using our method, we analyse the model $G_\alpha\cup \Gnp$ with respect to
the containment of spanning bounded degree graphs,
addressing a problem which was highlighted by
Krivelevich, Kwan and Sudakov in the concluding remarks of~\cite{krivelevich2015bounded}.
 We obtain the following result.
\begin{theorem}
\label{theorem:main}
Let $\alpha > 0$ be a constant, $\Delta \ge 5$ be an integer and $G_\alpha$
be a graph with minimum degree at least $\alpha n$.
Then, for every $F \in \cF(n, \Delta)$ and $p = \omega \left( n^{-2/(\Delta+1)} \right)$, whp $G_\alpha\cup \Gnp$ contains a copy of $F$.
\end{theorem}

Our bound on $p$ in Theorem~\ref{theorem:main} is best possible in the
following sense.
In the case where $F$ is a $K_{\Delta+1}$-factor on $n$ vertices
and $G_\alpha$ is a complete bipartite graph with parts of size $\alpha n$ and $(1-\alpha)n$, we need to find an almost
spanning $K_{\Delta+1}$-factor on $(1-\alpha (\Delta+1))n$ vertices in
$\Gnp$. This can easily be shown to require $p=\Omega \left(
  n^{-2/(\Delta+1)} \right)$. Note in addition that the edge probability used in Theorem~\ref{theorem:main} is lower by a $\log$-term in comparison to the anticipated threshold for the graph $F$ to appear in $\Gnp$~(see Conjecture~\ref{thresconj}).

\smallskip

Our second result deals with  powers of Hamilton cycles. Here we can
save a polynomial factor  $n^{\Omega(1)}$
compared to the threshold $n^{-1/k}$ in $\Gnp$.

\begin{theorem}\label{theorem:HC2}
	For each $k\geq 2$ and $\alpha>0$, there is some $\eta >0$, such that
    if $G_\alpha$ is an $n$-vertex graph with minimum degree at least
    $\alpha n$, then $G_\alpha\cup G(n,n^{-1/k-\eta})$ whp contains the $k$th power of a Hamilton cycle.
\end{theorem}

It was proved by Koml{\'{o}}s, S{\'{a}}rk{\"{o}}zy, and
Szemer{\'{e}}di~\cite{KSSkthHam} that $G_\alpha$ on its own contains the
$k$th power of a Hamilton cycle, provided that $\alpha\ge k/(k+1)$ and $v(G_\alpha)$
is large enough. Bedenknecht, Han, Kohayakawa and
Mota~\cite{BHKM_powers} showed that for any $k\ge 3$ there is an $\eta$ so
that $G_\alpha\cup G(n,n^{-1/k-\eta})$ whp contains the $k$th
power of a Hamilton cycle if $\alpha>c_k$ for some absolute constant
$c_k>0$.

Bennett, Dudek, and Frieze~\cite{BenDuFrPerturbed} gave the
following lower bound.
With $G_\alpha$ the complete bipartite graph with $\alpha n$ and
$(1-\alpha)n$ vertices in the classes,
one can show that~$p$ has to be at least $n^{-1/k(1-2\alpha)}$ for $G_\alpha \cup \Gnp$ to contain the $k$th power of a Hamilton cycle.
It would be interesting to determine the optimal dependence between $\alpha$, $k$ and $\eta$ in Theorem~\ref{theorem:HC2}.

\subsection*{Organisation}
We finish this section by providing some further notation, before outlining our general embedding method for randomly perturbed graphs in Section~\ref{sec:aa}.
We then prove Theorem~\ref{theorem:HC2}, the less technical of our implementations of this method, in Section~\ref{sec:HC2}.
Theorem~\ref{theorem:main} is proved in Section~\ref{sec:max}, with the proofs of two auxiliary lemmas given in Section~\ref{sec:proofLemmata}. Finally, we make some concluding remarks and sketch how our methods can give simpler proofs of other results in the literature concerning
randomly perturbed graphs in Section~\ref{sec:conclude}.

\subsection*{Notation}
A graph $G$ has vertex set $V(G)$, edge set $E(G)$, and we let $v(G)=|V(G)|$ and $e(G)=|E(G)|$. For a vertex $v\in V(G)$, $N_G(v)$ is the set of neighbours of $v$ in $G$, and for a vertex set $A\subset V(G)$, $N_G(A)=(\cup_{v\in A}N_G(v))\setminus A$. Where no confusion is likely to occur, we simply write $N(v)$ and $N(A)$. For graphs $G$ and $H$, $G\cap H$ is the graph on vertex set $V(G)\cap V(H)$ with edge set $E(G)\cap E(H)$. For a graph $G$, and a vertex set $A\subset V(G)$, $G[A]$ is the induced subgraph of $G$ on $A$, and $G-A=G[V(G)\setminus A]$.

\section{Tools}

Our results concern the embedding of certain graphs~$F$ in $G_\alpha\cup G(n,p)$.
For obtaining such an embedding, our first step will always be to embed an almost
spanning subgraph~$F^*$ of~$F$, and our second step then (working in an auxiliary graph on $[2n]$
vertices) extends this to an embedding of~$F$.

For the second step we shall use the following hypergraph matching theorem
of Aharoni and Haxell~\cite{aharoni2000hall}. The setup will be as
follows. $F\setminus F^*$ consists of~$t$ well-separated subgraphs
$S_1,\dots,S_t$ of~$F$, and we shall encode all valid embeddings
of~$S_i$ that extend the embedding of~$F^*$ as the edges of a hypergraph~$L_i$. The goal
then is to find a hypergraph matching using exactly one edge from
each~$L_i$. A hypergraph is \emph{$r$-uniform} if each of its edges has cardinality~$r$.

\begin{theorem}[Hall's condition for hypergraphs~\cite{aharoni2000hall}]
	\label{theorem:hall}
	Let $\{ L_1, \dots, L_t\}$ be a family of $s$-uniform hypergraphs on the same vertex set.
	If, for every $\cI \subseteq [t]$, the hypergraph $\bigcup_{i \in \cI} L_i$ contains a matching of size greater than $s(|\cI|-1)$, then there exists a function $g: [t] \rightarrow \bigcup_{i=1}^t E(L_i)$ such that $g(i) \in E(L_i)$ and $g(i) \cap g(j) = \emptyset$ for $i \not= j$.
\end{theorem}

When we want to use this theorem, we need to
verify the condition on~$L_i$. For this purpose
we shall use Janson's inequality (see, e.g., \cite[Theorem~2.18]{janson2011random}).

\begin{lemma}[Janson's inequality]
	\label{lemma:Janson}
	Let $p \in (0,1)$ and consider a family $\{ H_i \}_{i \in \cI}$ of subgraphs of the complete graph on the vertex set $[n]=\{1,\ldots,n\}$. For each $i \in \cI$, let $X_i$ denote the indicator random variable for the event that $H_i \subseteq \Gnp$ and, write $H_i \sim H_j$ for each ordered pair $(i,j) \in \cI \times \cI$ with $i \neq j$ if $E(H_i) \cap E(H_j) \not= \emptyset$.
	Then, for $X = \sum_{i \in \cI} X_i$, $\mathbb{E}[X] = \sum_{i \in \cI} p^{e(H_i)}$,
	\begin{align*}
	\delta = \sum_{H_i \sim H_j} \mathbb{E}[X_i X_j] = \sum_{H_i \sim H_j} p^{e(H_i) + e(H_j) - e(H_i \cap H_j)}
	\end{align*}
	and any $0 < \gamma < 1$ we have
	\begin{align*}
	\mathbb{P} [X \le (1-\gamma) \mathbb{E}[X]] \le \exp \left(-\frac{\gamma^2 \mathbb{E}[X]^2}{2(\mathbb{E}[X] + \delta)} \right).
	\end{align*}
\end{lemma}

This result will also be useful for the first step described
above, in which we embed an almost spanning subgraph. In particular,
the appearance of almost $H$-factors in $G(n,p)$ for $p\ge Cn^{-1/m_1(H)}$
is a straightforward consequence of Janson's inequality (see, e.g.,
\cite[Theorem~4.9]{janson2011random}). Here we need a minor modification of
this result. For two graphs~$H_1$ and~$H_2$, an \emph{$(H_1,H_2)$-factor} is any graph
that consists only of vertex disjoint copies of~$H_1$ and~$H_2$. The
following theorem concerning the appearance of an almost $(H_1,H_2)$-factors in
$G(n,p)$ can be proved with trivial modifications to the proof of~\cite[Theorem~4.9]{janson2011random}.

\begin{theorem}[Almost factors in $G(n,p)$]
\label{thm:factors}
  For every pair of graphs~$H_1$ and~$H_2$, and every $\eps>0$  there is a constant~$C$ such that,
  if $p\ge Cn^{-1/m_1(H_i)}$ for $i=1,2$, then for every $(H_1,H_2)$-factor $F^*$ on at most
  $(1-\eps)n$ vertices, whp $G(n,p)$ contains~$F^*$.
\end{theorem}

For our result on spanning bounded degree subgraphs we shall also use
the following result of Riordan~\cite{Riordan}, which allows the embedding
of spanning subgraphs that are not locally too dense
in $G(n,p)$. For a graph~$H$ let
\begin{align*}
  \gamma(H) = \max_{S \subseteq H, v(S) \ge 3} \frac{e(S)}{v(S)-2}\,.
\end{align*}
Riordan's Theorem can be found in the following form in~\cite{parczyk2016spanning}.
We shall use this theorem for a subgraph~$H$ of~$F$ which excludes the
`dense spots' of~$F$.

\begin{theorem}[Riordan~\cite{Riordan}]
\label{theorem:Riordan}
Let $\Delta \ge 2$ be an integer, $H \in \cF(n,\Delta)$ and $p= \omega( n^{-\frac{1}{\gamma(H)}} )$.
Then, a copy of $H$ is contained in $\Gnp$ whp.
\end{theorem}

Finally, we shall use the following submartingale-type inequality to handle
weak dependencies in the proof of our main technical result.
A proof of this lemma can for example be found in~\cite[Lemma 2.2]{allen2016blow}.

\begin{lemma}[Sequential dependence lemma]
\label{lemma:process}
Let $\Omega$ be a finite  probability space, and let $\cF_0,\dots,\cF_m$ be partitions of $	\Omega$, with $\cF_{i-1}$ refined by $\cF_i$ for each $i \in [m]$.
For each $i \in [m]$ let $Y_i$ be a Bernoulli random variable on $\Omega$ which is constant on each part of $\cF_i$.
Let $\delta$ be a real number, $\gamma \in (0,1)$, and $X= Y_1 + \dots + Y_m$.
If $\EE [Y_i | \cF_{i-1}] \ge \delta$ holds for all $i \in [m]$, then
\begin{align*}
\PP [X \le (1-\gamma) \delta m] \le \exp \left( \frac{-\gamma^2 \delta m}{3} \right) .
\end{align*}
\end{lemma}

\section{Main technical theorem}\label{sec:aa}

We start with an outline of the main idea of our strategy for embedding
some spanning graph~$F$ into $G_\alpha\cup G(n,p)$.  Recall that $G(n,p)$
has vertex set~$[n]$.  We use two-round exposure.  In the first round we
will find an $F^*$-copy for some almost spanning induced subgraph $F^*$
of~$F$. One key idea in our proof is that, by symmetry, the
$F^*$-copy we find is random among all possible $F^*$-copies in the
complete graph on vertex set~$[n]$ (see Section~\ref{sec:indep}).  Hence it remains to complete
such a random $F^*$-copy to an $F$-copy using only edges in $G_\alpha$ and
the second round (see Section~\ref{sec:comp}.).  It is the additional edges of~$G_\alpha$ in this
second step that allow us to gain on the bound for embedding~$F$ in a random
graph alone.

For the second round, we use an absorbing method, relying on the following
family of reservoir sets.

\begin{definition}[Reservoir sets]
  Given a graph $G_\alpha$ on vertex set~$[n]$, a copy $\hat F$ of a subgraph~$F^*$ of~$F$ in the complete graph on
  vertex set~$[n]$, and an independent set~$W$ of vertices of $\hat F$, we
  define the family of \emph{$(G_\alpha, \hat F, W)$-reservoir sets}
  $\big(R(u)\big)_{u\in[n]}$ by setting
  \begin{equation}\label{Bvdefn}
    R(u)=\big\{ w\in W\colon N_{\hat F}(w)\subset N_{G_\alpha}(u) \big\}\,.
  \end{equation}
\end{definition}

The crucial property of these reservoir sets is as follows. Assume that
$\hat F$ is a copy of~$F^*$ in $G(n,p)$. Then, for any vertex
$u\in[n]\setminus V(\hat F)$ exchanging~$u$ with any vertex $w\in R(u)$ gives us a different
copy of~$F^*$ in $G_\alpha\cup G(n,p)$, now using~$u$. In this case we also say that we can
\emph{switch} $u$ and~$w$.
Moreover, since~$W$ is an independent set, switching several vertices
simultaneously in this manner does not create conflicts.
As part of our proof we will show (see Lemma~\ref{lemma:Bsize}) that, for a
random $\hat F$ and a suitably chosen set~$W$, the sets $R(u)$ are likely
to have linear size intersections with neighbourhoods in $G_\alpha$.
This will give us `enough room' to complete~$\hat F$ to~$F$.

\smallskip

Next, we will state the technical embedding theorem, Theorem~\ref{newtheorem}, that formalises this method.
Theorem~\ref{theorem:HC2} and
Theorem~\ref{theorem:main} will be inferred from this result.
In our technical theorem we are given, along with~$F$, a family $\cF$ of
almost spanning subgraphs of~$F$. This family is chosen such that whp one of these subgraphs
appears in our first round and such that, in our second round whp each subgraph
in~$\cF$ can be extended to~$F$, using vertex switching. We call a set
$\cF$ with these properties suitable, defined formally as follows.

\begin{definition}[Suitability]\label{suitabledefn} Let $F$ be an $n$-vertex graph with maximum degree $\Delta$.
A set~$\cF$ of induced subgraphs of $F$ is called
\emph{$(\alpha,p)$-suitable} if, with
\begin{equation}\label{eq:eps}
  \eps =\Big(\frac{\alpha}{4\Delta}\Big)^{2\Delta}\,,
\end{equation}
each graph in $\cF$ has at least $(1-\eps)n$ vertices and the following two properties hold.
\begin{enumerate}[label=\itmit{A\arabic{*}}]
\item \label{AA1}
$\mathbb{P}(\exists F^*\in\cF\text{ with some }F^*\text{-copy in } G(n,p/2))=1-o(1)$.
\item \label{AA2}
Suppose $F^*\in \cF$ and $G$ is a graph with vertex set $[2n]$ which
contains a copy $\hat F$ of $F^*$.
For each $v\in V(F)\setminus V(F^*)$, let $B(v)\subset [2n]\setminus
V(\hat F)$ be a set such that $|B(v)\cap N_G(w)|\geq 4\eps n$  for each $w\in [2n]$.
Then whp $\hat F$ can be extended to a copy of~$F$ in
$G\cup G(2n,p/6)$
such that each vertex $v\in V(F)\setminus V(F^*)$ is mapped to a vertex in $B(v)$.
\end{enumerate}
\end{definition}

Observe that in~\ref{AA2} we consider auxiliary graphs on $[2n]$.
These encode all the information we need from $G_\alpha$ and our second
round of randomness. The sets $B(v)$ then are the corresponding auxiliary
versions of our reservoir sets. This setup, using $[2n]$, allows
us to keep the auxiliary reservoir sets disjoint from the $F^*$-copy. The
idea is, if $F^*$ can be extended to~$F$ in this auxiliary graph, then
this corresponds to a homomorphism of~$F$ in the original setting on $[n]$, and we can use switches to turn this homomorphism into an embedding.

We remark that in the proof of our first result, Theorem~\ref{theorem:HC2}
on squares of Hamilton cycles the family~$\cF$ only contains a single
graph. In the proof of Theorem~\ref{theorem:main}, however, the use of a
larger family is crucial.

\begin{theorem}[Main technical result]\label{newtheorem}
Let $\alpha>0$ and $\Delta\in\mathbb{N}$ be constant and let $p=p(n)$. If
$G_\alpha$ and~$F$ are $n$-vertex graphs such that
\begin{enumerate}[label=\rom]
\item $V(G_\alpha)=[n]$ and $\delta(G_\alpha)\ge\alpha n$,
\item $\Delta(F)=\Delta$ and
$F$ has an $(\alpha,p)$-suitable set of subgraphs $\cF$,
\end{enumerate}
then $G_\alpha\cup \Gnp$ whp contains a copy of $F$.
\end{theorem}

The main work for deducing our main results from this theorem will go into finding an $(\alpha,p)$-suitable family~$\cF$.
Verifying~\ref{AA1} corresponds to finding an almost spanning embedding for some $F^* \in \cF$, which is usually not too hard, because $\varepsilon n$ vertices remain uncovered.
To show~\ref{AA2}, by the definition of the $B(v)$ there is a linear number of options for the
embedding of every vertex, which makes this step again be somewhat similar to an
almost spanning embedding (and we can also use the edges of~$G$).

We will argue in Section~\ref{sec:conclude} that using this theorem we can
also easily derive short proofs for a number of related results from the literature.
We now turn to the proof of Theorem~\ref{newtheorem}.

%

\subsection{Reducing the problem to completing a random subgraph copy}\label{sec:indep}

In this section we show that, using two-round exposure and \ref{AA1}, we can reduce the
problem of embedding~$F$ in $G_\alpha\cup \Gnp$ to extending a random copy of an almost spanning subgraph.

\begin{lemma}\label{lem:reduction}
  Let $\alpha,\Delta,p$ and $G_\alpha$, $F$, and $\cF$ be as in the hypothesis of Theorem~\ref{newtheorem}.
  For each $F^*\in\cF$ let~$\hat F$ be a random $F^*$-copy in the complete
  graph on vertex set~$[n]$, and assume that
  \begin{equation}
    \label{eq:tobeshown}
    \mathbb{P}(\exists\text{ an $F$-copy in }G_\alpha \cup \hat F\cup G(n,p/2))=1-o(1)\,.
  \end{equation}
  Then $G_\alpha\cup \Gnp$ whp contains a copy of~$F$.
\end{lemma}

\begin{proof}
  Let $G_1$ and $G_2$ be two independent copies
  of $G(n,p/2)$.
  For finding a copy of~$F$ in $G(n,p)$, we want to use the edges of~$G_1$ to find a copy of $F^*\in\cF$,
  and then use~\eqref{eq:tobeshown} to complete such a copy
  to~$F$ using the edges of~$G_2$ and~$G_\alpha$.  For the second step we
  will condition on the success of the first step. For this purpose, we define the following events.  Let $F^*_1,\dots,F^*_r$ be the graphs in~$\cF$. For each
  $1\leq i\leq r$, let $\cE_i$ be the event that there is a copy of $F_i^*$
  in $G_1$, but no copy of $F_j^*$ for every $j<i$. Note that this event is
  empty if $F_j^*$ is a subgraph of $F_i^*$ for some $j<i$. These events
  are chosen such that
  \begin{equation}
  \label{eq:sumEi}
    \sum_{i=1}^r\mathbb{P}(\cE_i)
    = \mathbb{P}(\exists i\text{ with some $F^*_i$-copy in $G_1$})
    = 1-o(1)\,,
  \end{equation}
  where the second equality uses~\ref{AA1}.

  In order to use~\eqref{eq:tobeshown} in the second step, it
  is essential that we obtain a random copy of $F^*\in\cF$ in the first
  step. Here, the crucial observation is that for each $i\in[r]$ and a
  random $F_i^*$-copy $\hat{F}_i$ in the complete graph on vertex set~$[n]$ we have
  \begin{equation}
    \label{eq:indepmon}
    \mathbb{P}(\exists\text{ an $F$-copy in }G_\alpha \cup G_1\cup G_2|\cE_i)
    \geq \mathbb{P}(\exists\text{ an $F$-copy in }G_\alpha \cup \hat{F}_i\cup G_2)\,.
  \end{equation}
  Indeed, this follows from the fact that~$G_1$ is independent of
  $G_\alpha\cup G_2$, and that, if we condition on~$\cE_i$ then~$G_1$
  contains an $F_i^*$-copy by definition and by symmetry each possible $F_i^*$-copy
  is equally likely to appear in~$G_1$.
  It follows that
  \begin{multline*}
    \mathbb{P}\big(\exists\text{ an $F$-copy in }G_\alpha \cup G(n,p)\big)\ge
    \mathbb{P}(\exists\text{ an $F$-copy in }G_\alpha \cup G_1\cup G_2) \\
    \begin{aligned}
      & \geq \sum_{i=1}^r \,\mathbb{P}(\exists\text{ an $F$-copy in }G_\alpha \cup G_1\cup G_2|\cE_i)\cdot \mathbb{P}(\cE_i)
      \\
      &\geByRef{eq:indepmon} \sum_{i=1}^r\,\mathbb{P}(\exists\text{ an $F$-copy in }G_\alpha \cup \hat{F}_i\cup G_2)\cdot \mathbb{P}(\cE_i)
      \eqByRef{eq:tobeshown} (1-o(1))\cdot\sum_{i=1}^r\mathbb{P}(\cE_i)
      \eqByRef{eq:sumEi}
      1-o(1)\,,
    \end{aligned}
  \end{multline*}
  as desired.
\end{proof}

\subsection{Completing a random subgraph copy}
\label{sec:comp}

In this section we provide the proof of our main technical theorem,
Theorem~\ref{newtheorem}. By Lemma~\ref{lem:reduction} it remains to show
that whp we can complete a random $F^*$-copy into a copy of $F$.
For this we will choose a
large $2$-independent set $W$ in the $F^*$-copy, which has no neighbours outside $F^*$ (this is with respect to $F^*$ as a subgraph of $F$), construct the
according reservoir sets, and perform switches.
Recall that a set $W$ of vertices in a graph is
called \emph{$2$-independent}, if it is independent and no pair of distinct
vertices in~$W$ have a common neighbour.
The following lemma, whose proof we defer to the end of the
section, states that these reservoir sets are well-distributed with respect
to $G_\alpha$-neighbourhoods.

\begin{lemma}
	\label{lemma:Bsize}
    Let $\alpha,\Delta,p$ and $G_\alpha$, $F$, and $\cF$ be as in
    Theorem~\ref{newtheorem}. Let $F^*\in \cF$ and let~$W^*$ be a maximally
    $2$-independent set in~$F^*$, which has no neighbours outside $F^*$. Let $\hat{F}$ be a random copy of
    $F^*$ in the complete graph on vertex set~$[n]$ and~$W$ be the image of~$W^*$
    in~$\hat F$. Then whp the $(G_\alpha, \hat F, W)$-reservoir sets
    $\big(R(u)\big)_{u\in[n]}$ satisfy
    that for each $u,v \in[n]$ we have \[|N_{G_\alpha}(v)\cap R(u)| \ge
      4\eps n\,,\]
    where $\eps=(\frac{\alpha}{4\Delta})^{2\Delta}$, as in~\eqref{eq:eps}.
\end{lemma}

This lemma in particular implies that the sets $R(u)$ are
linear in size.

\begin{proof}[Proof of Theorem~\ref{newtheorem}]
  Assume we are given graphs~$G_\alpha$ and~$F$ satisfying the assumptions and a suitable set of
  almost spanning subgraphs~$\cF$ of~$F$. Fix $F^*\in\cF$ and let~$\hat F$
  be a random copy of~$F^*$ in the complete graph on vertex set~$[n]$ and
  let~$g_0$ be the embedding that maps~$F^*$ to~$\hat F$.

  By Lemma~\ref{lem:reduction}, it suffices to
  prove~\eqref{eq:tobeshown}.
  For this purpose we will use the reservoir
  sets and~\ref{AA2}. So, let $W^*$ be a maximally
  $2$-independent set in~$F^*$, which has no neighbours outside $F^*$, let~$W$ be the image of~$W^*$
  under~$g_0$, and let $\big(R(u)\big)_{u\in[n]}$ be the
  $(G_\alpha, \hat F, W)$-reservoir sets. By Lemma~\ref{lemma:Bsize}, whp, for all
  $u,v \in[n]$ we have $|N_{G_\alpha}(v)\cap R(u)| \ge 4\eps n$.

We now start by mapping the remaining vertices of~$F$ arbitrarily to the
unused vertices $[n]\setminus V(\hat F)$. Our goal then is to use
switchings to turn this mapping into an embedding of~$F$.
So, label the vertices in $[n]\setminus V(\hat{F})$ arbitrarily as $\{z_v\colon v\in V(F)\setminus V(F^*)\}$.
In order to appeal to~\ref{AA2} we now define an auxiliary graph~$G$
on vertex set~$[2n]$ together with a collection of auxiliary reservoir
sets $B(u)$, which encode the embedding~$g_0$ of~$F^*$ and the edges
of~$G_\alpha$ as well as the
reservoir sets $R(u)$.

Let $G$ be the auxiliary graph on the vertex set $[2n]$ that contains all edges
of~$\hat F$ in addition to exactly the following edges.
For each edge $uw$ of~$G_\alpha$ the graph~$G$ contains the edges
$\{u+n,w\}$, $\{u,w+n\}$, and $\{u+n,w+n\}$.
For each $v\in V(F)\setminus V(F^*)$, we define the auxiliary reservoir set
$B(v)=\{w+n\colon w\in R(z_v)\}$.
Since $|N_{G_\alpha}(v)\cap R(u)| \ge 4\eps n$ for all $u,v \in[n]$, we have
for each $v \in V(F)\setminus V(F^*)$ and $w \in [2n]$ that $|B(v)\cap N_G(w)|\geq 4\eps n$.
So the graph~$G$ and the sets $B(v)$ fit the setup in~\ref{AA2}.

Now let $G'_2$ be a copy of $G(2n,p/6)$ on vertex set $[2n]$.
Hence, by~\ref{AA2} the following event~$\cE$ holds whp:
$\hat F$ can be extended to a copy
of~$F$ in $G\cup G'_2$ such that each $v\in V(F)\setminus V(F^*)$ is
mapped to~$B(v)$. The corresponding embedding~$g'$ of~$F$ into $G\cup G'_2$
extends~$g_0$. In particular, this $F$-copy in the auxiliary graph encodes
which vertices get switched where (as we detail below).

Now we need to translate this back to our original setting
on~$n$ vertices. For this, let~$G_2$ be the graph on vertex set $[n]$ and
with all edges $uw$ such that $\{u, w+n\}$, $\{u+n, w\}$ or $\{u+n, w+n\}$ is an edge in $G'_2$.
Hence~$G_2$ is distributed as a random graph in which each edge appears independently and
with probability at most $p/2$. Therefore, in order to
show~\eqref{eq:tobeshown} it is sufficient to prove that whenever the
event~$\cE$ holds for~$G'_2$, then there also is an $F$-copy in $G_\alpha \cup \hat{F}\cup G_2$.

Indeed, assume that~$\cE$ holds and define for each $v\in V(F)$
\[
g(v)= \begin{cases}
\, g'(v)-n & \text{ if }v\in V(F)\setminus V(F^*),\\
\,z_u & \text{ if }g'(v)=g'(u)-n\text{ for some } u \in V(F)\setminus V(F^*),\\
\,g'(v) & \text{ otherwise}\,.\\
\end{cases}
\]
In other words, the first line states that all vertices~$v$ in $V(F)\setminus V(F^*)$, which by the
definition of~$B(v)$ are embedded by~$g'$ in $[2n]\setminus[n]$, are mapped by~$g$ to the corresponding vertex in~$[n]$.
The third line guarantees that vertices~$v$ in $V(F^*)$ usually are
embedded by~$g'$ as by~$g$, unless this creates a conflict with the rule
from the first line for a vertex~$u$, in which case they are switched
to~$z_u$ by the second line.

We claim that $g$ is an embedding of $F$ into $G_\alpha \cup \hat{F}\cup G_2$.
To see this, let
\[
Z_0=V(F)\setminus V(F^*)\text{\,\, and \,\,}Z_1=\{v: g'(v)=g'(u)-n\text{ for some } u \in Z_0 \}.
\]
Note that $g$ agrees with $g'$ outside of $Z_0\cup Z_1$, so that $g$
(appropriately restricted) is an embedding of $F-(Z_0\cup Z_1)$ into $G_\alpha \cup \hat{F}$.
Now consider any $v\in Z_1$ and let $u \in Z_0=V(F)\setminus V(F^*)$ be
such that $g'(v)=g'(u)-n$. Since~$u$ is embedded by~$g'$ into
$B(u)=\{w+n\colon w\in R(z_u)\}$, we have
$g'(v)=g'(u)-n\in R(z_u)$. Recall that
$R(z_u)\subset W$ by the definition of the reservoir sets, and~$W$ is the
image under $g'$ of~$W^*$.
We conclude that
$Z_1\subset W^*$, that is, $Z_1$ is $2$-independent and has no
neighbours outside~$F^*$. It follows
that vertices in $Z_1$ have no $F$-neighbours in $Z_0$ or $Z_1$.
Thus, for each $v\in Z_1$,
\[
g\big(N_F(v)\big)=g'\big(N_F(v)\big)=N_{\hat{F}}\big(g'(v)\big)\subset N_{G_\alpha}(z_u)\,,
\]
where the last step uses $g'(v)\in R(z_u)$.
This shows that vertices in~$Z_1$ are properly embedded by~$g$.

It remains to consider vertices $v \in Z_0$. We prove that all neighbours
of~$v$ are mapped to neighbours of $g(v)$, distinguishing three cases.
Firstly, for $u\in N_{F}(v)\setminus ( Z_0 \cup Z_1)$,
there is an edge between $g(v)=g'(v)-n$ and $g(u)=g'(u)$ in $G_\alpha \cup
G_2$, because there is an edge between $g'(v)$ and $g'(u)$
in $G \cup G'_2$.
Secondly, for $u\in N_{F}(v)\cap Z_0$, there is an edge between
$g(v)=g'(v)-n$ and $g(u)=g'(u)-n$ in $G_\alpha \cup G_2$, because there is
an edge between $g'(v)$ and $g'(u)$ in $G \cup G'_2$. Finally,
$N_{F}(v)\cap Z_1$ is empty, because vertices in~$Z_1$ do not have any
$F$-neighbours in~$Z_0$.

We conclude that~$g$ is an embedding of~$F$ into
$G_\alpha \cup \hat{F}\cup G_2$, completing the proof of
Theorem~\ref{newtheorem}.
\end{proof}

It remains to prove Lemma~\ref{lemma:Bsize}, which is based on the fact
that the reservoir sets~$R(u)$ are
{random} sets.

\begin{proof}[Proof of Lemma~\ref{lemma:Bsize}]
  Note that, as $F$ has maximum degree at most $\Delta$, we have
  \[
    |W^*|\geq (|F^*|-\Delta|V(F)\setminus V(F^*)|)/\Delta^2\geq n/(2\Delta^2).
  \]
  Let~$g_0$ be the (random) mapping of~$F^*$ to~$\hat F$, and observe that, by symmetry, $W=g_0(W^*)$
  is a uniformly random set of size~$|W^*|$ in~$[n]$.

  Fix $u,v\in V(G_\alpha)$. For each $w^*\in W^*$,  note that $|N_{F^*}(w^*)|\leq \Delta$ and that
  the sets $\{w^*\}\cup N_{F^*}(w^*)$ are all disjoint.  Let $I_{w^*}$
  be the indicator variable for the event $g_0(w^*)\in N_{G_\alpha}(v)$ and
  $N_{\hat{F}}\big(g_0(w^*)\big)\subseteq N_{G_\alpha}(u)$.
  Since by definition $R(u)=\{w\in W\colon N_{\hat{F}}(w)\subset
  N_{G_\alpha}(u)\}=\{g_0(w^*)\colon w^*\in W^*, N_{\hat{F}}\big(g_0(w^*)\big)\subseteq N_{G_\alpha}(u)\}$,
  it follows that
  \begin{equation}\label{yesthisone}
    |N_{G_\alpha}(v)\cap R(u)|=\sum_{w^*\in W^*}I_{w^*}\,.
  \end{equation}
  Let $r= \frac{\alpha n}{3\Delta^2}\leq |W|$ and pick distinct vertices
  $w^*_1,\ldots, w^*_r$ in $W^*$. Consider revealing the random copy~$\hat{F}$ by,
  firstly, revealing the mapping of vertices in $\{w^*_1\}\cup N_{F^*}(w^*_1)$, then revealing the
  mapping of vertices in $\{w^*_2\}\cup N_{F^*}(w^*_2)$, and so on, until $\{w^*_r\}\cup N_{F^*}(w^*_r)$,
  before finally revealing the rest of the vertices in $\hat{F}$. Note that, for each
  $1\leq i\leq r$, when the location of the vertices in $\{w^*_i\}\cup N_{\hat{F}}(w^*_i)$
  is revealed there are at least $\alpha n/2$ vertices both in
  $N_{G_\alpha}(u)$ and $N_{G_\alpha}(v)$ which are not occupied by a vertex
  in $\{w^*_j\}\cup N_{F^*}(w^*_j)$ with $j<i$. Hence, for each $1\leq i\leq r$, if
  $m=|N_{F^*}(w^*_i)|$ and $\cH_i$ is the history of the location of the
  vertices in $\{w^*_j\}\cup N_{F^*}(w^*_j)$ with $j<i$, then
  \begin{equation}\label{otherone}
    \mathbb{E}(I_{w^*_i}|\cH_i)\geq \frac{\alpha n/2\cdot \binom{(\alpha
        n/2)-1}{m}}{n\binom{n}{m}}\ge\left(\frac{\alpha}{4}\right)^{m+1}
   \geq \Big(\frac{\alpha}{4}\Big)^{\Delta+1}\,.
  \end{equation}
  Therefore, by~\eqref{yesthisone} and Lemma~\ref{lemma:process} applied
  with $\delta=(\frac{\alpha}{4})^{\Delta+1}$, we have
  \[
    |N_{G_\alpha}(v)\cap R(u)|\geq 3\delta r/4\geq \frac{\alpha^{\Delta+2}n}{4^{\Delta+2}\Delta^2}
    \geq 4\Big(\frac{\alpha}{4\Delta}\Big)^{2\Delta}n=4\eps n\,,
  \]
  with probability $1-\exp(-\Omega(\delta r))=1-o(n^{-2})$.
  Using a union bound, we conclude that with probability $1-o(1)$ for each
  $u,v\in V(G_\alpha)$ we have $|N_{G_\alpha}(v)\cap R(u)|\geq 4\eps n$.
\end{proof}


\section{Powers of Hamilton Cycles}
\label{sec:HC2}

Let $F=C^{(k)}_n$ be the $k$th power of the cycle with $n$ vertices, and let
$P^{(k)}_n$ denote the $k$th power of a path with $n$ vertices.
To prove Theorem~\ref{theorem:HC2}, it is sufficient, by
Theorem~\ref{newtheorem} to find an $\eta=\eta(\alpha)>0$, such that there
exists an $(\alpha,p)$-suitable set~$\cF$ of subgraphs of~$F$ with \[p=n^{-1/k-\eta}\,.\]
In fact, we will use only one subgraph, which will consist of disjoint copies
of the $k$th power of long (but constant length) paths, which we connect by
shorter $k$th powers of paths to form a copy of~$F$.

In the following we shall explain how we choose~$\cF$, and show that~$\cF$
satisfies~\ref{AA1} and~\ref{AA2} for $p=n^{-1/k-\eta}$, which implies that $\cF$ is
$(\alpha,p)$-suitable. We use the following constants.
Given~$k$ and~$\alpha>0$, let $\Delta=2k$ and $\eps = (\frac{\alpha}{4 \Delta})^{2\Delta}$.
Pick large integers $m$ and $\ell$, and a small constant $\eta>0$ such that
\[
\alpha , \frac{1}{k}\gg \frac{1}{\ell}\gg\frac{1}{m}  \gg \eta>0\,,
\]
where, for example, by $\frac{1}{m}  \gg \eta$ we mean that the following
proof works if we choose $\eta$ sufficiently small compared to $1/m$. In
particular, we require $\ell^2\le\eps m$.

\subsection{Choosing \texorpdfstring{$\cF$}{F}}

Let $\cF$ solely contain $F^*$, the following
$(P^{(k)}_m, P^{(k)}_{m+1})$-factor on at least $(1-\eps)n$ vertices,
which is a subgraph of~$F$.  Let~$s$ and~$t$ be the unique integers such
that $n=s(m+\ell)+t$ and $t<(m+\ell)$. Let~$F^*$ be the graph on
$v(F^*)=sm+t=t(m+1)+(s-t)m$ vertices consisting of the following vertex
disjoint $k$th powers of paths: $t$ copies of $P_{m+1}^{(k)}$, which we
denote by $P_1^*,\dots,P_{t}^*$, and $s-t$ copies of $P_{m}^{(k)}$, which
we denote by $P_{t+1}^*,\dots,P_s^*$. This leaves exactly
$v(F)-v(F^*)=s\ell\le s\eps m\le\eps n$ vertices of~$F$ uncovered.

Observe that we obtain~$F$ from~$F^*$ by connecting for each $i\in[s]$ the
paths $P_i^*$ and $P_{i+1}^*$ (respectively $P_1^*$ if $i=s$)
by a $k$th power of a path  with $\ell$ vertices, which we denote by $w_{i,1}^*,\dots,w_{i,\ell}^*$, such that the following is satisfied.
For $i\in[s]$ let $u^*_{i,1},\dots,u^*_{i,k}$ be the end $k$-tuple of
$P_i^*$ and $v^*_{i,1},\dots,v^*_{i,k}$ be the start $k$-tuple of
$P_{i+1}^*$ (respectively $P_1^*$ if $i=s$). We require that
\[u^*_{i,1},\ldots,u^*_{i,k},w^*_{i,1},\ldots,w^*_{i,\ell},v^*_{i,1},\ldots,v^*_{i,k}\]
is the $k$th power of a path with $\ell+2k$ vertices.


\subsection{Proof that~\texorpdfstring{$\cF$}{F} satisfies \texorpdfstring{\ref{AA1}}{(A1)}}
We use Theorem~\ref{thm:factors} to find a copy of~$F^*$ in $G(n,p/2)$.
Since for $m'\ge 2k$ we have $e(P^{(k)}_{m'})=km'-\binom{k+1}{2}$,
it is easy to check that for $k\ge 2$ we have $m_1(P^{(k)}_m),m_1(P^{(k)}_{m+1})<k$. Since~$F^*$ is an $(P^{(k)}_m, P^{(k)}_{m+1})$-factor on at
most $(1-\eps)n$ vertices, it follows directly from Theorem~\ref{thm:factors} that
$G(n,p/2)$ contains a copy of~$F^*$, and hence \ref{AA1} holds for $\cF$.

\subsection{Proof that~\texorpdfstring{$\cF$}{F} satisfies \texorpdfstring{\ref{AA2}}{(A2)}}\label{sec:second_stage}
Suppose that $G$ is a graph with vertex set $[2n]$ which contains a copy~$\hat F$ of~$F^*$.
For each $v\in V(F)\setminus V(F^*)$, assume we are given a set $B(v)\subset [2n]\setminus V(\hat
F)$ such that for each $w\in [2n]$ we have $|B(v)\cap N_G(w)|\geq 4\eps n$. Let $G'=G(2n,p/6)$.
Our goal is to extend~$\hat F$ to a copy of~$F$ in $G\cup G'$ such that
each vertex~$v$ in $V(F)\setminus V(F^*)$ is mapped to $B(v)$.

For each $i\in [s]$ and $j\in[k]$, let $u_{i,j}$ be the image of
$u^*_{i,j}$ in~$\hat F$, and $v_{i,j}$ be the image of $v^*_{i,j}$ in~$\hat F$.
Hence, to extend $\hat F$
to a copy of $F$ we need to embed all vertices $w^*_{i,j}$ with $i\in[s]$ and
$j\in [\ell]$,
to distinct vertices $w_{i,j}$ so that
\begin{equation}\label{kpath}
u_{i,1},\ldots,u_{i,k},w_{i,1},\ldots,w_{i,\ell},v_{i,1},\ldots,v_{i,k}
\end{equation}
is the $k$th power of a path with $2k+\ell$ vertices.

We would like to appeal to Hall's condition for hypergraphs, Theorem~\ref{theorem:hall}, to show that this is
possible. For this purpose we define the following auxiliary hypergraphs.
Let $W=[2n]\setminus V(\hat{F})$. For each $i\in [s]$, let $L_i$ be the
$\ell$-uniform hypergraph with vertex set $W$ where $e\in \binom{W}{\ell}$
is an edge exactly if there is some ordering of $e$ as
$w_{i,1},\ldots,w_{i,\ell}$ so that~\eqref{kpath} is the $k$th power of a
path in $G\cup G'$ and $w_{i,j}\in B(w^*_{i,j})$ for each $j\in[\ell]$.
We shall argue that the following lemma, whose proof we defer to
Section~\ref{extendnewproof}, guarantees that the assumption of
Theorem~\ref{theorem:hall} is satisfied.

\begin{lemma}\label{extendnew} For each $r\in [s]$ and $A\subset [s]$ with $|A|=r$ and $U\subset W$ with $|U|\leq \ell^2r$, the following holds with probability at least $1-\exp(-\omega(r\log n))$. There exists some $i\in A$ and an edge $e\in E(L_i)$ with $V(e)\subset W\setminus U$.
\end{lemma}

The property in Lemma~\ref{extendnew} fails for some $r\in [s]$ and $A\subset [s]$ with $|A|=r$ and $U\subset W$ with $|U|\leq \ell^2r$ with probability at most
\[
\sum_{r\in [s]}\binom{s}{r}\binom{2n}{\ell^2 r} \cdot\exp(-\omega(r\log n))=o(1),
\]
so we may assume the property holds for all such sets.

To apply Theorem~\ref{theorem:hall}, we need to show that, for every $A \subseteq [s]$, the hypergraph $\bigcup_{i\in A} L_i$ contains a matching with size greater than $\ell(|A|-1)$.
Indeed, let $A\subseteq [s]$ and $r=|A|$, and let $U$ be the vertex set of a
maximal matching in $\bigcup_{i\in A} L_i$. This means that there is no
$i\in A$ and edge $e\in E(L_i)$ with $V(e)\subseteq
W\setminus U$. Thus, by the property from
Lemma~\ref{extendnew}, we have $|U|\geq \ell^2|A|$, so that
$\bigcup_{i\in A} L_i$ contains a matching with size at least
$\ell|A|$. Therefore, we can apply Theorem~\ref{theorem:hall},
and obtain a function
$\pi:[s]\to\bigcup_{i\in [s]}E(L_i)$ such that $\pi(i)\in E(L_i)$ for each
$i\in [s]$ and the edges in $\pi([s])$ are vertex disjoint.
Observe that, by the definition of the hypergraphs~$L_i$, embedding
the vertices $w^*_{i,1},\ldots,w^*_{i,\ell}$ to the vertices of $\pi(i)$ in
an appropriate order yields the desired completion of~$\hat F$ to an
embedding of~$F$. Thus, subject only to the proof of
Lemma~\ref{extendnew},~\ref{AA2} holds as required.

\subsection{Proof of Lemma~\ref{extendnew}}\label{extendnewproof} We will
prove Lemma~\ref{extendnew} using Janson's inequality,
Lem\-ma~\ref{lemma:Janson}. Recall that the hyperedges of each
hypergraph~$L_i$ represent legitimate connections in $G\cup G'$ between the
images of the $k$th power of paths
$P_i^*$ and $P_{i+1}^*$ in $\hat F$.

\begin{proof}[Proof of Lemma~\ref{extendnew}]
Fix $r\in [s]$ and $A\subset [s]$, $U\subset W$ with $|A|=r$ and $|U|\leq \ell^2r\leq \ell^2s\leq \eps n$. Let $j =\lfloor \ell/2\rfloor$.
Let $P$ be the $k$th power of the path with vertex set
\begin{equation}\label{Porder}
u_{1},\ldots,u_{k},w_{1},\ldots,w_{\ell},v_{1},\ldots,v_{k},
\end{equation}
with all the edges between the vertices $u_i$ removed and all the edges between the vertices $v_i$ removed. Furthermore, remove from $P$ the edges $u_kw_1$, $w_{\ell}v_1$ and all the edges $w_iw_{i+1}$, $i\in [\ell-1]$, except for $w_j w_{j+1}$. The edges that we have removed will come from the deterministic graph $G$, while we will find a copy of $P$ in $G'$. The edge $w_j w_{j+1}$ is included in $P$ so that we do not need to find a path between $v_{k}$ and $w_1$ in $G$.

To simplify our calculations for the application of Janson's inequality let us first prove three simple claims concerning the density of subgraphs of $P$. Let $U=\{u_{1},\ldots,u_{k}\}$ and $V=\{v_{1},\ldots,v_{k}\}$.
\begin{claim}\label{claimpath0} $e(P)\leq \ell(k-1/2)$.
\end{claim}
\begin{claimproof}[Proof of Claim~\ref{claimpath0}] In the ordering of the
  vertices in $P$ in \eqref{Porder}, ignoring the edge $w_jw_{j+1}$,
  each vertex has at most $k-1$ neighbours to the right. Therefore,
  including the edge $w_jw_{j+1}$, we have $e(P)\leq (\ell+2k)(k-1)+1\leq \ell(k-1/2)$, since we chose $\ell\gg k$.
\end{claimproof}

\begin{claim}\label{claimpath2}
For each subgraph $P'\subset P-(U\cup V)$ with $e(P')\geq 1$, we have $p^{- e(P')}\cdot n^{1-v(P')}=o(\log^{-1} n)$.
\end{claim}
\begin{claimproof}[Proof of Claim~\ref{claimpath2}] Removing the edge $w_j w_{j+1}$ if necessary, we have that each vertex in $P'$ has at most $(k-1)$ neighbours to the right in the labelling in~\eqref{Porder}. As the rightmost vertex in $P'$ has no such neighbours, if $v(P')\geq 3$, then
we have
$e(P')\leq (k-1)(v(P')-1)+1\leq  k(v(P')-1)-1$. If $v(P')= 2$, then $e(P')= 1\leq k(v(P')-1)-1$.
Therefore, as $\eta\ll 1/\ell,1/k$,
\[
p^{-e(P')}\cdot n^{1-v(P')}\leq p(p^kn)^{1-v(P')}=n^{-1/k-\eta}(n^{-k\eta})^{1-v(P')}\leq n^{-1/k-\eta+k\ell\eta}=o(\log^{-1}n).\hfill\qedhere
\]
\end{claimproof}

\begin{claim}\label{claimpath}
For each subgraph $P'\subset P$ with $P'\neq P$, $e(P')\geq 1$ and $U, V\subset V(P')$, we have $p^{- e(P')}\cdot n^{2k-v(P')}=o(\log^{-1} n)$.
\end{claim}
\begin{claimproof}[Proof of Claim~\ref{claimpath}] For such a subgraph $P'$, let $W_0=V(P')\setminus (U\cup V)$.
 We enumerate the vertices from $W_0$ by $w_{i_1}$, \ldots, $w_{i_t}$ from left to right in the ordering~\eqref{Porder}. If there is an index  $a$ with $i_{a+1}-i_a\ge k+1$, then we estimate the number of edges in $P'$ through $(k-1)|W_0|+1$. This is so because we can enumerate all the edges of $P'$ by  identifying at least one vertex adjacent to every edge of $P'$ as follows: every vertex $w_{i_c}$ ($c\le a$) is adjacent to the left to at most $k-1$ vertices, and every vertex $w_{i_c}$ ($c> a$) is adjacent to the right to at most $k-1$ vertices, the only exception being possibly the vertices $w_j$ and $w_{j+1}$ along with the edge $w_j w_{j+1}$, thus contributing one more possible edge. Therefore, if $|W_0|\geq 2$, then
\[
e(P')\leq (k-1)|W_0|+1\leq k|W_0|-1\leq k(v(P')-2k)-1.
\]
If $|W_0|=1$, then, as $\ell\gg k$ the vertex in $W_0$ cannot have neighbours in both $U$ and $V$, so that $e(P')\leq (k-1)= k(v(P')-2k)-1$.

If there is no such index $a$ as above, then
note that $v(P')\geq |W_0|\geq (\ell-k)/k\gg k^2$. Then, counting from the edges of $P'$ from their leftmost vertex in \eqref{Porder}, and remembering that $w_j w_{j+1}$ may be an edge, $e(P')\leq (k-1)v(P')+1\leq k(v(P')-2k)+2k^2-v(P')+1\leq k(v(P')-2k)-1$. Thus, in all cases, $e(P')\leq k(v(P')-2k)-1$.

Therefore, as $\eta\ll 1/\ell$,
\[
p^{-e(P')}\cdot n^{2k-v(P')}\leq p(p^kn)^{2k-v(P')}=n^{-1/k-\eta}(n^{-k\eta})^{2k-v(P')}\leq n^{-1/k-\eta+k\ell\eta}=o(\log^{-1}n).\hfill\qedhere
\]
\end{claimproof}

For each $i\in [A]$, let $\mathcal{P}_i$ be the set of copies of $P$ in the graph $G'$ with vertices in order (to match~\eqref{Porder})
\[
u_{i,1},\ldots,u_{i,k},w_{i,1},\ldots,w_{i,\ell},v_{i,1},\ldots,v_{i,k}
\]
where $w_{i,1}\in B(w^*_{i,1})\cap N_G(u_{i,k})$,  $w_{i,j'+1}\in B(w^*_{i,j'+1})\cap N_G(w_{i,j'})$ for each $j'\in [j-1]$,  $w_{i,\ell}\in B(w^*_{i,\ell})\cap N_G(v_{i,1})$, and $w_{i,j'-1}\in B(w^*_{i,j'-1})\cap N_G(w_{i,j'})$ for each $j'\in \{j+1,\ldots,\ell\}$.
That is, if such a copy of $P$ exists in $G'$ then the edge $\{w_{i,1},\ldots,w_{i,\ell}\}$ is in~$L_i$.

Note that, choosing the vertices in order $w_{i,1},\ldots,w_{i,j}, w_{i,\ell},w_{i,\ell-1},\ldots,w_{i,j+1}$, there are at least $4\eps n-|U|-s\ge 2\eps n$ options for each vertex, and therefore $|\mathcal{P}_i|=\Omega(n^{\ell})$.
Let $\mathcal{P}=\cup_{i\in A}\mathcal{P}_i$, so that $|\mathcal{P}|=\Omega(r\cdot n^{\ell})$.

For each $Q,Q'\in \mathcal{P}$, with $Q\neq Q'$, let $Q\sim Q'$ if $Q$ and $Q'$ share some edge. Let $q=p/6$, the edge probability in $G'$. Denote the expectation for the number of graphs from $\cP$ in $G'$ by $\mu=|\mathcal{P}|q^{e(P)}$ and let
\begin{align*}
\delta = \sum_{Q,Q'\in \mathcal{P}\colon\, Q\sim Q'}q^{2e(P) - e(Q \cap Q')}.
\end{align*}
 Note that, as $\eta\ll 1/k$ and $p=n^{-1/k-\eta}$, we have $q^{(k-1/2)}n=\omega(\log n)$.
As $|\mathcal{P}|=\Omega(r\cdot n^{\ell})$, we then have, using Claim~\ref{claimpath0},
\[
\mu=\Omega(r\cdot q^{e(P)}n^{\ell})=\Omega(r\cdot (q^{(k-1/2)}n)^{\ell})
=\omega(r\cdot \log n).
\]

Recall that for $j\in [k]$ the vertices $u_{i,j}$ and $v_{i,j}$ denote  the images of the end-$k$-tuples of the graph $P_i^*$ from $F^*$  given through the copy $\hat F$ of $F^*$ and moreover that all these $k$-tuples contain distinct vertices (cf.\ Section~\ref{sec:second_stage}). Let $i\in A$ and $Q\in \mathcal{P}_i$, and let $U_i=\{u_{i,1},\ldots,u_{i,k}\}$ and $V_i=\{v_{i,1},\ldots,v_{i,k}\}$. For each $P'\subsetneq Q$ with $U_i,V_i\subset V(P')$, there are at most $n^{\ell+2k-v(P')}$ graphs $Q'\in \mathcal{P}$ with $Q\cap Q'=P'$
(all of which are in $\mathcal{P}_i$).  For each subgraph $P'\subset Q-(U_i\cup V_i)$, there are at most $r\cdot n^{\ell-v(P')}\leq n^{\ell+1-v(P')}$ graphs $Q'\in \mathcal{P}$ with $Q\cap Q'=P'$.
Thus,
\[
\delta\leq |\mathcal{P}|\cdot \left(
\sum_{P'\subset P-(U\cup V)\colon\,e(P')\geq 1}q^{2e(P) - e(P')} n^{\ell+1-v(P')}
+
\sum_{P'\subsetneq P:U, V\subset V(P')}q^{2e(P) - e(P')}n^{\ell+2k-v(P')}
\right).
\]
Therefore, as  $|\mathcal{P}|=O(r n^{\ell})$ and $\mu=\Omega(r q^{e(P)}n^{\ell})$,
\[
\frac{\delta r}{\mu^2}=O\left(
\sum_{P'\subset P-(U\cup V):e(P')\geq 1}q^{- e(P')}\cdot n^{1-v(P')}
+
\sum_{P'\subsetneq P:U, V\subset V(P')}q^{-e(P')}n^{2k-v(P')}
\right)=o(\log^{-1}n),
\]
using Claims~\ref{claimpath2} and~\ref{claimpath}.
Thus, as $\mu =\omega(r\log n)$ and $\frac{\delta}{\mu^2}=o(r^{-1}\log^{-1}
n)$ we can infer from Janson's inequality, Lemma~\ref{lemma:Janson}, that with probability at least $1-\exp(-\omega(r\log n))$ there is some $i\in A$ and $Q\in \mathcal{P}_i$ in $G'$, and hence $V(Q)\setminus (U_i\cup V_i )\in E(L_i)$, as required.
\end{proof}

\section{Spanning subgraphs with bounded maximum degree}\label{sec:max}
Let $F\in \cF(n,\Delta)$ and $p=\omega(n^{-\frac{2}{\Delta +1}})$. As
before, we find a suitable set $\cF$ of large subgraphs of $F$ such that we
can whp embed one of these subgraphs $F^*\in \cF$ in
$G(n,p/2)$ (\ref{AA1} in Definition~\ref{suitabledefn}), and then extend any such $F^*$-copy (in an
auxiliary graph) to cover all of $F$ (\ref{AA2} in Definition~\ref{suitabledefn}). To do this, we adapt the strategy of Ferber, Luh and Nguyen~\cite{ferber2016embedding} to decompose~$F$.
In~\cite{ferber2016embedding}, each graph $F\in \cF(n,\Delta)$ is decomposed into a sparse part and many dense spots. Our set $\cF$ will consist of subgraphs of $F$ covering the sparse part and most of the dense parts.

Recall that the parameter
	\begin{align*}
		\gamma(H) = \max_{S \subseteq H, v(S) \ge 3} \frac{e(S)}{v(S)-2}
	\end{align*}
determines when we can apply Riordan's theorem, Theorem~\ref{theorem:Riordan}, to
embed a spanning subgraph in $G(n,p)$.
In the following we call a graph $H$ \emph{dense} if $\gamma(H) > \tfrac{\Delta+1}{2}$ and
\emph{sparse} otherwise.
We can now define, following~\cite{ferber2016embedding}, a good decomposition of a graph.

	\begin{definition}[$\varepsilon$-good decomposition]
		\label{def:good}
		Let $\varepsilon>0$, $F \in \cF(n,\Delta)$ and let $\cS_1,\dots, \cS_k$ be families of induced subgraphs of~$F$.
		For $F^\prime = F - (\bigcup_h \bigcup_{S \in \cS_h} V(S))$ we say that $(F^\prime,\cS_1,\dots,\cS_k)$ is an \emph{$\varepsilon$-good decomposition} if the following hold.
		\begin{enumerate}[label=\itmit{P\arabic{*}}] 
			\item $F^\prime$ is sparse, that is, $\gamma(F^\prime) \le \tfrac{\Delta+1}{2} $. \label{prop:sparse}
			\item Each $S \in \bigcup_h\cS_h$ is minimally dense, that is, $\gamma(S) > \tfrac{\Delta+1}{2} $ and $S^\prime$ is sparse for all $S^\prime \subseteq S$ with $3\le v(S^\prime) < v(S)$. \label{prop:dense}
			\item For each $1 \le h \le k$, all the graphs in $\cS_h$ are isomorphic.	\label{prop:isom}

			\item Every $\cS_h$ contains graphs on at most $\varepsilon n$ vertices, that is $|\bigcup_{S \in \cS_h} V(S)| \le \varepsilon n$. \label{prop:small}

			\item All the graphs in $\bigcup_i \cS_i$ are vertex disjoint and, for each $1 \le h \le k$ and $S,S^\prime \in \cS_h$ with $S\neq S'$, there are no edges between $S$ and $S'$ in $F$, and $S$ and $S'$ share no neighbours in $F$.\label{prop:disjoint}
		\end{enumerate}
        We call the graphs in $\cS_1,\dots, \cS_k$ the \emph{dense spots}
        of the decomposition.
	\end{definition}

We remark that our definition is slightly less restrictive than that from~\cite{ferber2016embedding}, where \itmit{P3} is replaced by a stronger condition.
	An $\varepsilon$-good decomposition can easily be found using a greedy algorithm.
	The following lemma is proved in~\cite{ferber2016embedding}.

	\begin{lemma}[Lemma~2.2 in~\cite{ferber2016embedding}]
		\label{lemma:decompose}
		For each $\varepsilon >0$ and $\Delta>0$, there exists some $k_0$ such that, for each $F \in \cF(n,\Delta)$, there is some $k\leq k_0$ and an $\varepsilon$-good decomposition $(F^\prime,\cS_1,\dots,\cS_k)$ of $F$.
	\end{lemma}

    In the following we shall use this lemma to define a family~$\cF$ of
    subgraphs of $F\in\cF(n,\Delta)$. We shall then show that this family~$\cF$ satisfies \ref{AA1} and
    \ref{AA2} and hence is $(\alpha,p)$-suitable, which by
    Theorem~\ref{newtheorem} implies
    Theorem~\ref{theorem:main} as desired.

\subsection{Choosing \texorpdfstring{$\cF$}{F}}
Fix $F \in \cF(n,\Delta)$.
Let $\eps = (\frac{\alpha}{4 \Delta})^{2\Delta}$, and let $k_0$ be large enough for the result of Lemma~\ref{lemma:decompose} to
hold with $\eps$ and $\Delta$.
By Lemma~\ref{lemma:decompose}, for some $k\leq k_0$,
there is an $\varepsilon$-good decomposition $(F^\prime,\cS_1,\dots,\cS_k)$ of
$F$, which we fix.

For each $1\leq h\leq k$, let $s_h$ be the size of the graphs in $\cS_h$ (possible by~\ref{prop:isom}), and, picking some representative $S\in \cS_h$, note that, by~\ref{prop:dense} and as $\Delta(S)\leq \Delta$, we have
\[
(\Delta+1)(s_h-2)<2e(S)\leq \Delta s_h,
\]
so that $s_h<2\Delta+2$. Thus, we may consider $\alpha$, $\Delta$, $\eps$, $k\leq k_0$, and the maximum size of each dense spot ($2\Delta+1$) to be constant, while $n$ tends to infinity.

Let $\cF$ contain exactly those induced subgraphs of $F$ which cover $F'$ and, for each $1\leq h\leq k$, all but at most $\frac{\eps n}{s_h^2k}$ of the graphs from $\cS_h$.

\subsection{Proof that~\texorpdfstring{$\cF$}{F} satisfies \texorpdfstring{\ref{AA1}}{(A1)}}
\label{sec:embedmost}

We shall embed the copy of $F^\prime$ using Riordan's theorem, Theorem~\ref{theorem:Riordan}.
In~\cite{ferber2016embedding}, the embedding of $F^\prime$ is then extended step
by step to include the graphs in $\cS_h$, for $1\leq h\leq k$. We proceed
similarly, but in each step only include most of the graphs $\cS_h$, for
$1\leq h\leq k$. This allows us to work at a lower probability than that
used in~\cite{ferber2016embedding}, as we aim to find a copy of only some graph in $\cF$.

To find such a copy of a graph in $\cF$, we expose the graph $G(n,p/2)$ in a total of $k+1$ rounds, revealing $G_h \sim G(n,q)$ for $0 \le h \le k$, where $q=p/(6k)$ and thus $(1-q)^{k+1} \ge 1-p/2$.
Every edge is thus present with probability at most $p/2$ in $\bigcup_hG_h$.
We use $G_0$ to embed $F'$ and then iteratively use $G_1,\ldots,G_k$ to
embed as many subgraphs from $\cS_1,\ldots,\cS_k$ as possible, and show
that this results whp in an embedding of a subgraph from $\cF$.

Since, by~\ref{prop:sparse}, $\gamma(F^\prime) \le \tfrac{\Delta+1}{2}$,
and thus $q = \omega( n^{-\frac{1}{\gamma(F^\prime)}})$, by
Theorem~\ref{theorem:Riordan}, we can whp embed $F^\prime$ into $G_0$.
Let $f_0 : V(F^\prime) \rightarrow V(G_0)$ be such an embedding and let $F_0'=f_0(F')$.

For $1 \le h \le k$, we want to (whp) use edges from $G_h$ to
extend the embedding $f_{h-1}$ to cover all but at most $\tfrac{\varepsilon
  n}{s_h^2 k}$ graphs from $\cS_h$. We then let $f_h$ be the extended
embedding and let $F'_h$ be the subgraph of $F$ embedded by $f_h$. We use
the following lemma, which allows us to extend the current embedding to one
more dense spot $S\in\cS_h$, even if we restrict its image to a small but
linearly sized set~$U$, using only edges of~$G_h$.
This lemma is proved along with another lemma from this section in Section~\ref{sec:proofLemmata}.

\begin{lemma}
	\label{lemma:proofalmost}
For each $1\leq h\leq k$, the following holds whp for any $\cS \subseteq \cS_h$ and $U \subseteq V(G_\alpha)$ with $|\cS | \ge \tfrac{\varepsilon n}{s_h^2 k} $ and $|U| \ge \tfrac{\varepsilon n}{s_hk}$.
There is some $S\in \cS$ and a copy $S'$ of $S$ in $G_h[U]$ with an
embedding $\pi: V(S)\to V(S')$ such that, for each $v\in V(S)$,
\begin{equation}\label{thisone}
f_{h-1}(N_{F}(v)\cap V(F'_{h-1}))\subseteq N_{G_h}(\pi (v)).
\end{equation}
\end{lemma}

Start with $f_0$ and $F'_0$. For each $1\leq h\leq k$, we construct $f_h$ and
$F'_h$, as follows. The property in Lemma~\ref{lemma:proofalmost} whp holds for $h$. We extend the embedding $f_{h-1}$ to $f_h$ using edges
from $G_h$ to cover as many of the graphs in $\cS_h$ as possible (with any
edges to $F'_{h-1}$ correctly embedded), and call the resulting graph
$F'_h$. By the property in Lemma~\ref{lemma:proofalmost}  this leaves at most $\tfrac{\eps n}{s_h^2k}$ graphs in $\cS_h$
unembedded. Indeed, if there is a set $\cS$ of at least $\tfrac{\eps
  n}{s_h^2k}$ unembedded graphs in $\cS_h$, then, let
$U=V(G_\alpha)\setminus V(F'_h)$ and note that $|U|\geq s_h\cdot |\cS|\geq
\tfrac{\eps n}{s_hk}$. There then exists some $S\in \cS$ and a copy $S'$ of
$S$ in $G_h[U]$ with isomorphism $\pi: V(S)\to V(S')$ such
that~\eqref{thisone} holds for each $v\in V(S)$. As, by~\ref{prop:disjoint}, no two subgraphs in $\cS_h$ have an edge between them, $\pi$ can be used to embed $S$ and extend the embedding $f_h$, a contradiction.

From this we obtain (whp) the embedding $f_k$ of a subgraph of $F$, covering $F'$ and all but at most $\tfrac{\eps n}{s_h^2k}$ graphs from each $\cS_h$, $1 \le h \le k$, into $\bigcup_h G_h$. Such a subgraph embedded by $f_k$ is thus in $\cF$, and therefore \ref{AA1} holds.


\subsection{Proof that~\texorpdfstring{$\cF$}{F} satisfies \texorpdfstring{\ref{AA2}}{(A2)}}
Let $F^*\in \cF$ and let the graph $G$ be as described in \ref{AA2} containing the copy $\hat F$ of $F^*$.

 For each $1\leq h\leq k$,
let $\cS_h^\prime \subseteq \cS_h$ be those dense parts not in $F^*$, so
that $|\cS'_h|\leq \tfrac{\eps n}{s_h^2k}$. We have, for each $1\leq h\leq
k$, that the graphs in $\cS'_h$ are isomorphic, minimally dense,
disjoint, and neither have edges between them nor share any neighbours. Furthermore,
the sets in $\{V(F^*)\}\cup \{V(S): S\in \cS'_h,1\leq h\leq k\}$ form a partition of $V(F)$. Note that $|V(F)\setminus V(F^*)|\leq \eps n$. For each $0\leq h\leq k$, let $F_h$ be the induced subgraph of $F$ with vertex set $V(F^*)\cup(\cup_{h'\leq h}\cup_{S\in \cS_{h'}}V(S))$.

Let $G'_1,\ldots, G'_k$ be independent random graphs with $G'_i\sim G(2n,q)$, where $q=p/(6k)$.
Starting with $g_0$ and $F_0=F^*$, for each $1 \le h \le k$ in turn, we
will (whp) inductively find a function
\[
g_h:V(F_h)\to [2n]
\]
such that
\begin{enumerate}[label=\itmit{Q\arabic{*}}] 
\item $g_h$ is an embedding of $F_h$ into $G\cup(\bigcup_{h'\leq h}G'_{h'})$, which extends $g_{h-1}$ and \label{gprop1}
\item \label{gprop2} for each vertex $v\in F_{h}\setminus F_{h-1}$, we have $g_h(v)\in B(v)$.
\end{enumerate}
Note that $g_0$ satisfies these properties, and that, once we find $g_k$ whp, we will have an embedding of $F_k=F$ into $G \cup  (\bigcup_{1\le h\le k} G'_{h})$, satisfying the conditions in \ref{AA2}. Noting that each edge in $\bigcup_hG_h$ appears independently at random with probability at most $p/6$, we then have that \ref{AA2} holds.

Suppose then that $1\leq h\leq k$ and we have found the function
$g_{h-1}$ satisfying~\ref{gprop1} and~\ref{gprop2}. Let $W_{h-1}=[2n]\setminus g_{h-1}(F_{h-1})$. 
For each $S\in \cS'_h$, label $V(S)=\{z_{S,1},\ldots, z_{S,s_h}\}$,
and let $L_S$ be the $s_h$-uniform auxiliary hypergraph with vertex set
$W_{h-1}$, where $e$ is an edge of $L_S$ if, for some labelling
$e=\{w_{S,1},\ldots,w_{S,s_h}\}$, the map $z_{S,i}\mapsto w_{S,i}$ is an
embedding of $S$ into $G \cup G'_h$, where, for each $1\leq i\leq s_h$ we have
$w_{S,i}\in B(z_{S,i})$ and $g_{h-1}(N_{F_h}(z_{S,i})\cap(V(F_{h-1})))\subset N_{G\cup G'_h}(w_{S,i})$.
Each hyperedge $e=\{w_{S,1},\ldots,w_{S,s_h}\}$ of $L_S$ then corresponds to a possible extension of $g_{h-1}$ to cover $S\in \cS'_h$.

We wish to show that whp there exists a function $\pi : \cS'_h \mapsto \bigcup_{S\in \cS_h'} E(L_S)$ such that $\pi(S)\in E(L_S)$ for each $S\in \cS_h'$, and the edges in $\pi(\cS'_h)$ are pairwise vertex disjoint. This is possible, as shown below, using Theorem~\ref{theorem:hall} and the following lemma.

\begin{lemma}
	\label{lemma:CompleteEmbedding}
	For each $1 \le h \le k$, $1 \le r \le |\cS_h'|$, $\cS \subseteq
    \cS'_h$ and $U\subseteq W_{h-1}$, with $|\cS|=r$ and $|U|\leq s_h^2r$,
    the following holds with probability at least $1 - \exp (- \omega(r \log
    (\tfrac{n}{r})))$. There exists some $S\in \cS$ and an edge $e\in E(L_S)$ with $V(e)\subseteq W_{h-1}\setminus U$.
\end{lemma}

The property in Lemma~\ref{lemma:CompleteEmbedding} then holds for each $1 \le h \le k$, $1 \le r \le |\cS'_h|$, $\cS \subseteq \cS'_h$ and $U\subseteq W_{h-1}$, with $|\cS|=r$ and $|U|\leq s_h^2r$ with probability at least
\[
1-k\cdot n\cdot \sum^{|\cS_h'|}_{r=1}\binom{n}{r}\cdot \binom{n}{s^2_hr}\cdot\exp(-\omega(r\log (\tfrac{n}{r})))=1-o(1).
\]
Similarly to our deductions from Lemma~\ref{extendnew}, it then follows that, for every $\cS \subseteq \cS_h'$, the hypergraph $\bigcup_{S \in \cS} L_S$ contains a matching with size greater than $s_h(|\cS|-1)$.
Therefore, by Theorem~\ref{theorem:hall}, a function $\pi$ as described
above exists. Thus, we can extend $g_{h-1}$ to an embedding $g_h$ of $F_h$
satisfying~\ref{gprop1} and~\ref{gprop2} as required.

Subject to the proof of Lemma~\ref{lemma:CompleteEmbedding}, this completes the proof that \ref{AA2} holds.


\section{Proofs of auxiliary lemmas}

\label{sec:proofLemmata}

In this section, we give the proofs of the lemmas from Section~\ref{sec:max}.

\subsection{Proof of Lemma~\ref{lemma:proofalmost}}
\label{sec:proofalmost}

We prove Lemma~\ref{lemma:proofalmost} with Janson's inequality, using similar calculations to Ferber, Luh and Nguyen~\cite{ferber2016embedding}.

\begin{proof}[Proof of Lemma~\ref{lemma:proofalmost}]
Fixing $h$, note that there are certainly at most $2^n\cdot 2^n$ choices for $\cS$ and $U$. Therefore, it is sufficient to prove, for fixed $\cS \subseteq \cS_h$ and $U \subseteq V(G_\alpha)\setminus f_{h-1}(V(F'_{h-1}))$ with $|\cS| \geq \tfrac{\varepsilon n}{s_h^2k}$ and $|U|\geq \tfrac{ \varepsilon n}{s_hk}$, the property in the lemma holds with probability $1-e^{-\omega(n)}$.

Let $s=s_h$.
Pick some $S_0\in \cS$, so that, by~\ref{prop:isom}, each graph in $\cS$ is isomorphic to $S_0$, and label $V(S_0)=\{v_1,\ldots,v_s\}$. Let $\cH$ be a set of $\binom{|U|}{s}$ copies of $S_0$ in the complete graph with vertex set $U$, where each copy of $S_0$ has a different vertex set. Note that $|U|=\Omega(n)$ and $|\cH|=\Omega(n^s)$. For each $S\in \cS$ and $H\in \cH$, label $V(S)=\{z_{S,1},\ldots,z_{S,s}\}$ and $V(H)=\{v_{H,1},\ldots,v_{H,s}\}$ so that $v_i\mapsto z_{S,i}$ and $v_i\mapsto v_{H,i}$ are embeddings of $S_0$.

Each graph in $\cS$ is isomorphic to $S_0$ in $F$, but, when we come to extend an embedding of $F'_{h-1}$ to $F'_h$ by embedding `most' of the copies from $\cS_h$, the number of edges between a copy from $\cS_h$ and the  already embedded $F'_{h-1}$ may differ.
We now distinguish two cases: Case I where each copy $S$ from $\cS$ has some edge between $S$ and $F'_{h-1}$ in $F'_h$ and Case II where there is some copy $S$ from $\cS$ for which there is no such edge.

Let us assume first that we are in Case I. For each $S\in \cS$, let
$W_S=f_{h-1}(\bigcup_{v\in V(S)}N_{F}(v)\cap V(F'_{h-1}))$ be the images of the
already embedded neighbours of vertices in~$S$. Note that these sets~$W_S$
are non-empty by the definition of Case I
and by~\ref{prop:disjoint} are disjoint.  For each $H\in \cH$ and $S\in \cS$,
let $H\oplus W_S$ be the graph with vertex set $V(H)\cup W_S$ containing
exactly those edges that we need in order to extend the partial
embedding we have to embed $S$  into~$H$. That is, $H\oplus W_S$ has
edge set
\[
E(H)\cup \{v_{H,i}v: 1\leq i\leq s, v\in f_{h-1}(N_{F}(z_{S,i})\cap V(F'_{h-1}))\}.
\]
For each $S\in \cS$, $H\in \cH$ and $J\subseteq H$, let $J\oplus
W_S=(H\oplus W_S)[V(J)\cup W_S]$. Let $\cH^+=\{H\oplus W_S: H\in \cH,S\in
\cS\}$, and note that if any graph from $\cH^+$ appears in $G_h$ then we
can indeed extend our current embedding to one more dense spot in~$\cS$,
and hence are done.

Let $\cJ=\{H\cap H': H,H'\in \cH,e(H\cap H')>0\}$ and $\cJ'=\{H\cap
H': H,H'\in \cH, H\neq H'\}\setminus \emptyset$. We will show that $\mathbb P(\exists H\in \cH^+\text{ with }H\subseteq G_h)=1-\exp(-\omega(n))$ follows from Lemma~\ref{lemma:Janson} and the following claim, which we then prove.

\begin{claim} \label{claim0} \mbox{} 
		\begin{enumerate}[label=\rom]
		\item \label{claim1} For each $J\in \cJ$, $2e(J)< (\Delta+1)(v(J)-1)$.
		\item \label{claim3} For each $H\in \cH$ and $S\in \cS$, $2e(H\oplus W_S)\leq (\Delta+1)s$.
		\item \label{claim2} For each $J\in \cJ'$ and $S\in \cS$, $2e(J\oplus W_S)< (\Delta+1)v(J)$.
	\end{enumerate}
\end{claim}

Note that, by~\ref{claim3} of Claim~\ref{claim0},
each graph in $\cH^+$ has at most $(\Delta + 1)s/2$ edges. We will now consider a subfamily $\cS'$ of size at least $\tfrac{2|\cS|}{(\Delta + 1)s}=\Omega(n)$ of those copies of $S$ from $\cS$ so that $S\oplus W_S$ has the same number of edges, say $m$, where  $1\le m\le (\Delta + 1)s/2$.
Using~\ref{claim3} of Claim~\ref{claim0}, and that $q=\omega(n^{-\frac{2}{\Delta+1}})$, let the expected
number of copies from $\cH^+_{\cS'}=\{H\oplus W_S: H\in \cH,S\in
\cS'\}$ in $G_h$ be denoted by  $\mu$, where
	\begin{align*}
	\mu=\sum_{S\in \cS'}\sum_{H\in \cH} q^{e(H\oplus W_S)}=\Omega(n^{s+1}q^{m})=\Omega(n^{s+1}q^{(\Delta+1)s/2})=\omega(n).
	\end{align*}
 Let
	\begin{equation}\begin{split}
	\delta &= \sum_{S,S'\in \cS'} \sum_{\substack{H,H' \in \cH \\ H \oplus W_{S} \sim  H' \oplus W_{S'}}} q^{e(H \oplus W_{S}) + e(H' \oplus W_{S'}) - e((H \oplus W_{S}) \cap (H' \oplus W_{S'}))}
 \\
&=q^{2m}\sum_{S,S'\in \cS'} \sum_{\substack{H,H' \in \cH \\ H \oplus W_{S}
    \sim  H' \oplus W_{S'}}} q^{- e((H \oplus W_{S}) \cap (H' \oplus
  W_{S'}))} \\
&\le q^{2m}\sum_{J\in \cJ}\sum_{\substack{S,S'\in \cS' \\ S\neq S'}}\sum_{\substack{H,H' \in \cH \\ H \cap H'=J }}q^{- e(J)}
+q^{2m}\sum_{J\in \cJ'}\sum_{S\in\cS'}\sum_{\substack{H,H' \in \cH \\ H \cap H'=J}}q^{- e(J \oplus W_S)}
\\
&\leq q^{2m}\sum_{J\in \cJ}|\cS'|^2n^{2s-2v(J)}q^{- e(J)}
+q^{2m}\sum_{J\in \cJ'}\sum_{S\in\cS'}n^{2s-2v(J)}q^{- e(J \oplus W_S)}.
\label{hope2}
\end{split}
\end{equation}
Then, using~\ref{claim1} and~\ref{claim2} of Claim~\ref{claim0}, and as $\mu=\Omega(n^{s+1}q^m)$, we have
\begin{align*}
\frac{\delta }{\mu^2}&=O\Big(\sum_{J\in \cJ}|\cS'|^2 n^{-2v(J)-2}q^{- e(J)}
+\sum_{J\in \cJ'}\sum_{S\in\cS'}n^{-2v(J)-2}q^{- e(J \oplus W_S)}\Big)
\\
&=O\Big(\sum_{J\in \cJ}n^{-2v(J)}q^{- (\Delta+1)(v(J)-1)/2}
+\sum_{J\in \cJ'}|\cS'|\cdot n^{-2v(J)-2}q^{- (\Delta+1)v(J)/2}\Big)
\\
&=o\Big(\sum_{J\in \cJ}n^{-2v(J)}n^{v(J)-1}
+\sum_{J\in \cJ'}n^{-2v(J)-1}n^{v(J)}\Big)=o(n^{-1}).
\end{align*}
Therefore, as $\mu=\omega(n)$ and $\frac{\delta}{\mu^2}=o(n^{-1})$, by Lemma~\ref{lemma:Janson}, the probability that there is no graph in $\cH^+_{\cS'}$ in $G_h$ is at most $\exp(-\frac{\mu^2}{4(\mu+\delta)})=\exp(-\omega(n))$, as required.
For Case I, it is left then only to prove Claim~\ref{claim0}.

\begin{claimproof}[Proof of Claim~\ref{claim0}]

For~\ref{claim1} let $H\in \cH$ be such that $J\subset H$. If $J\neq H$, and $v(J)\geq 3$, then, by~\ref{prop:dense}, we have $2e(J)\leq (\Delta+1)(v(J)-2)<(\Delta+1)(v(J)-1)$, as required. If $v(J)=2$, then $(\Delta+1)(v(J)-1)= \Delta+1> 2\geq 2 e(J)$.

Suppose then that $|J|=|H|$, so $v(J)=s$. If $s\leq\Delta$, then $2e(J)\leq
s(s-1)< (s+1)(s-1)\leq (\Delta+1)(s-1)$, and if $s>\Delta+1$, then
$2e(J)\leq s\Delta<s\Delta+s-(\Delta+1)= (\Delta+1)(s-1)$, as required. If
$s=\Delta+1$, note that, as there is some edge between $S_0$ and $F_{h-1}$
in $F_h$, we have that $S_0$, and hence $J$, is not a clique with $\Delta+1$ vertices. Thus, $2e(J)< s(s-1)=(\Delta+1)(s-1)$.

\smallskip

For~\ref{claim3} suppose $s\geq \Delta+1$. As $H$ is dense we have $2e(H)>(\Delta+1)(s-2)$, and thus
\[
2e(H\oplus W_S)\leq 2\Delta s-2e(H)< 2\Delta s-(\Delta+1)(s-2)=(\Delta+1)s+2(\Delta+1-s)\leq (\Delta+1)s,
\]
as required.

So suppose that $s\leq \Delta$. If $4\leq s\leq \Delta-1$, then, as $2e(H)>(\Delta+1)(s-2)$, we must have
\[
s(s-1)>(\Delta+1)(s-2)\geq (s+2)(s-2)=s(s-1)+s-4\geq s(s-1),
\]
a contradiction. If $s=3$, then $2e(H)>\Delta+1$ contradicts $\Delta\geq 5$.

Finally, if $s=\Delta$, then~$H$ must be the clique on $\Delta$ vertices
because $2e(H)>(\Delta+1)(\Delta-2)=\Delta(\Delta-1)-2$. Therefore,
\[
2e(H\oplus W_S)\leq 2\Delta^2-2e(H)=\Delta(\Delta+1)=(\Delta+1)s.
\]

\smallskip

For~\ref{claim2} let $H, H'\in \cH$ be such that $H\cap H'=J$ and $H\neq
H'$, which exist by the definition of~$\cJ'$.
Observe that $v(J)< s$ (since by our choice of $\cH$  the vertex sets of any two copies are distinct). Let $I=H-V(J)$, and let $e(I,J)$ be the number of edges between $I$ and $J$ in $H$. Then,
\begin{multline*}
2e(J\oplus W_S)\leq 2(\Delta v(J)-e(J)-e(J,I))=2(\Delta v(J)-e(H)+e(I)) \\ =(\Delta+1)v(J)+(\Delta-1)v(J)-2e(H)+2e(I).
\end{multline*}
Thus, to prove the claim it is sufficient to show that $(\Delta-1)v(J)< 2e(H)-2e(I)$.

As $H$ is dense, we have $2e(H)> (\Delta+1)(v(J)+v(I)-2)$. If $v(I)\geq 3$, then, from~\ref{prop:dense}, we have $2e(H)> (\Delta+1)v(J)+2e(I)$. If $v(I)= 2$, then $2e(H)>(\Delta+1)v(J)\geq (\Delta-1)v(J)+2e(I)$.

Finally, suppose $v(I)=1$, so that $e(I)=0$ and $v(J)=s-1$. By the reasoning in the proof of~\ref{claim3}, $s\geq \Delta$, otherwise we reach a contradiction. Thus, $2e(H)>(\Delta+1)(s-2)\geq (\Delta-1)(s-1)= (\Delta-1)v(J)-2e(I)$.

In each case, then, $2e(H)-2e(I)>(\Delta-1)v(J)$ as required.
\end{claimproof}

It remains to consider Case II. In this case there is some graph from $\cS\subseteq\cS_h$ with no edges to $F'_{h-1}$. Therefore, it is sufficient for some graph in $\cH$ to exist. Let $m=e(S_0)$ be the size of each (isomorphic) graph in $\cH$, and note that $2m\leq \min\{s\Delta,s(s-1)\}\leq (s-1)(\Delta+1)$. Thus, we may take
\[
\mu=\sum_{H\in \cH}q^m=\Omega(n^sq^m)=\Omega(n^sq^{(s-1)(\Delta+1)/2})=\omega(n).
\]

Let $\cJ=\{H\cap H': H,H'\in \cH, e(H\cap H')>0,H\neq H'\}$ and note that, if $J\in \cJ$ and $v(J)\geq 3$, then $2e(J)\leq (\Delta+1)(v(J)-2)$ by~\ref{prop:dense}. Let
\begin{align*}
\delta &=\sum_{\substack{H,H' \in \cH \\ H\sim  H',H\neq H' }} q^{e(H) + e(H') - e(H \cap H' )}
=q^{2m}\sum_{J\in \cJ}\sum_{\substack{H,H' \in \cH \\ H\cap  H'=J }}q^{-e(J)}
\leq q^{2m}\sum_{J\in \cJ}n^{2s-2v(J)}q^{- e(J)}
\\
&\leq q^{2m-1}n^{2s-2}+q^{2m}\sum_{J\in \cJ: v(J)\geq 3}n^{2s-2v(J)}q^{-(\Delta+1)(v(J)-2)/2}.
\end{align*}
Then, as $\mu=\Omega(n^sq^m)$, we have
\[
\frac{\delta}{\mu^2}=O\Big(q^{-1}n^{-2}+\sum_{J\in \cJ: v(J)\geq 3}n^{-2v(J)}q^{-(\Delta+1)(v(J)-2)/2}\Big)
=o\Big(n^{-1}+\sum_{J\in \cJ: v(J)\geq 3}n^{-v(J)-2}\Big)=o(n^{-1}).
\]
Therefore, as $\mu=\omega(n)$, and $\frac{\delta}{\mu^2}=o(n^{-1})$, by Lemma~\ref{lemma:Janson}, the probability that there is no graph in $\cH$ in $G_h$ is at most $\exp(-\frac{\mu^2}{2(\mu+\delta)})=\exp(-\omega(n))$, as required.
\end{proof}


\subsection{Proof of Lemma~\ref{lemma:CompleteEmbedding}}

Again, we use Janson's inequality and similar calculations to Ferber, Luh and Nguyen~\cite{ferber2016embedding}.

\begin{proof}[Proof of Lemma~\ref{lemma:CompleteEmbedding}]
Recall that
$W_{h-1}=[2n]\setminus g_{h-1}(F_{h-1})$.
Let $1\leq h\leq k$, $1\leq r\leq |\cS_h'|\leq \tfrac{\eps n}{s_h^2k}$, $\cS\subseteq \cS_h'$ and $U\subseteq W_{h-1}$ with $|\cS|=r$ and $|U|\leq s_h^2r$. Note that, as $|U|\leq s_h^2r\leq \eps n$, we have $|U\cup (W_{0}\setminus W_{h-1})|\leq 2\eps n$.
Therefore, by the property stated in \ref{AA2}, for each $v\in V(F)\setminus V(F^*)$ and each $u\in [2n]$, we have
\begin{equation}\label{helpful}
|N_{G}(u,B(v))\cap (W_{h-1}\setminus U)|\geq 2\eps n,
\end{equation}
 and, in particular, $|B(v)\cap (W_{h-1}\setminus U)|\geq 2\eps n$ and we set $B'(v)=B(v)\cap (W_{h-1}\setminus U)$.

Let $s=s_h$. 
 As in the proof of Lemma~\ref{lemma:proofalmost}, we will consider two cases: Case I where each copy $S$ from $\cS$ has some edge between $S$ and $F_{h-1}$ in $F_h$ and Case II when for some copy $S$ from $\cS$ there is no such edge.

\smallskip

Suppose first
we are in Case~I. For all $S\in \cS$, since $\Delta(F)\le \Delta$ and $|S|=s$, there are certainly at most $\Delta s$ vertices in $F_{h-1}$ with some edge in $S$, and at most $2^s$ ways of attaching such a vertex to $S$. Thus, we can consider a subfamily $\cS'$ of at least $\tfrac{1}{2^{\Delta s^2}}|\cS|=\Omega(n)$ copies of $S$ from $\cS$ which are all {isomorphic} when the edges from $S$ to $F_h$ are added to $S$.
 Pick $S_0\in \cS'$.
Label $V(S_0)=\{v_1,\ldots,v_s\}$ so that $v_1$ has a
neighbour in $F_{h-1}$.
Recall that for $S\in \cS$ we labelled $V(S)=\{z_{S,1},\ldots, z_{S,s_h}\}$.
Without loss of generality, we can assume for each $S\in \cS$ that $v_i\mapsto z_{S,i}$ is an isomorphism from $S_0$ into $S$, and that $z_{S,1}$ has a neighbour in $F_h$ in $V(F_{h-1})$ (possible as we are in Case I).  Let $\cH$ be a set of $\binom{|U|}{s}$ copies of $S_0$ in the complete graph with vertex set $U$, where each copy of $S_0$ has a different vertex set.
For each $H\in \cH$, label $V(H)=\{v_{H,1},\ldots,v_{H,s}\}$ so that
$v_i\mapsto v_{H,i}$ is an isomorphism of~$S_0$ to~$H$.

For each $S\in \cS'$, pick the image~$w_S$ of an already embedded neighbour
of the vertex~$z_{S,1}$ corresponding to~$v_1$, that is, pick
$w_S\in
g_{h-1}(N_{F_h}(z_{S,1})\cap V(F_{h-1}))$.
For each $S\in \cS'$, let
\[
\cH_S=\{H\in \cH: v_{H,1}\in N_{G}(w_S)\text{ and }v_{H,i}\in B'(v_{S,i})\text{ for each }1\leq i\leq s\}.
\]
For each $S\in \cS'$, note that, from~\eqref{helpful}, we have $|\cH_S|=\Omega(n^s)$.

For each $S\in \cS'$, let $W_S=g_{h-1}(\bigcup_{v\in V(S)}N_{F_{h}}(v)\cap V(F_{h-1}))$ be
the set of images of already embedded neighbours of vertices in~$S$.
For each $H\in \cH_S$ and $S\in \cS'$, let $H\oplus W_S$
be the graph with vertex set $V(H)\cup W_S$ and edge set
\[
E(H)\cup (\{v_{H,i}v: 1\leq i\leq s, v\in g_{h-1}(N_{F_{h}}(w_{S,i})\cap V(F_{h-1}))\}\setminus \{v_{H,1}w_S\}).
\]
These are exactly the edges we need in order to extend our embedding of $F_{h-1}$ to contain $S$ embedded into $H$, as $v_{H,1}w_S\in E(G)$.
Let $\cH^+=\{H\oplus W_S: S\in \cS',H\in \cH_S\}$, and note that if any graph $H\oplus W_S\in \cH^+$ appears in $G_h'$ then, as $v_{H,1}w_S\in E(G)$, $V(H)\in E(L_S)$, and we are done.

Let $\cJ=\{H\cap H': H,H'\in \cH,e(H\cap H')>0\}$ and $\cJ'=\{H\cap
H': H,H'\in \cH, H\neq H'\}\setminus \emptyset$. Note that~\ref{claim1} and~\ref{claim2} of Claim~\ref{claim0} hold here as well. For each $H\in \cH$ and $S\in \cS'$, $E(H\oplus W_S)$ does not include $v_{H,1}w_S$, and, therefore, in place of~\ref{claim3}, the following holds.
\emph{
		\begin{enumerate}[label=\romprime,start=2]
			\item \label{claim4} For each $S\in \cS$ and $H\in \cH_S$, $2e(H\oplus W_S)\leq (\Delta+1)s-2$.
		\end{enumerate}
}
Note that, by our choice of $\cS'$, each graph in $\cH^+$ has the same number of edges, $m$ say. Note that, as the property we are looking for is monotone, we may assume that $q^{-1/2}=\omega(\log n)$. Using~\ref{claim4}, let
	\begin{align*}
	\mu = \sum_{S\in \cS'}\sum_{H\in \cH_S} q^{m}=\Omega(rn^{s}q^m)=\Omega(rn^{s}q^{(\Delta+1)s/2-1})=\Omega(rq^{-1})=\omega(r\log n).
	\end{align*}
We remark that this is the only place where we use that the edge $v_{H,1}w_S$ is not included in $H \oplus W_S$, since it is already present in $G$.

Defining $\delta$ as follows, and using similar deductions to those used to reach~\eqref{hope2}, we have
\begin{align*}
\delta &= \sum_{S,S'\in \cS'} \sum_{\substack{H\in\cH_S,H' \in \cH_{S'} \\ H \oplus W_{S} \sim  H' \oplus W_{S'}}} q^{e(H \oplus W_{S}) + e(H' \oplus W_{S'}) - e((H \oplus W_{S}) \cap (H' \oplus W_{S'}))}
\\
 &\leq q^{2m}r^2\sum_{J\in \cJ}n^{2s-2v(J)}q^{- e(J)}
+q^{2m}\sum_{J\in \cJ'}\sum_{S\in\cS'}n^{2s-2v(J)}q^{- e(J \oplus W_S)}.
\end{align*}

Then, using~\ref{claim1} and~\ref{claim2} of  Claim~\ref{claim0}, and that $\mu =\Omega(rn^{s}q^m)$, we have
\begin{align*}
\frac{\delta }{\mu^2}&=O\Big(\sum_{J\in \cJ}n^{-2v(J)}q^{- e(J)}
+r^{-2}\sum_{J\in \cJ'}\sum_{S\in\cS'}n^{-2v(J)}q^{- e(J \oplus W_S)}\Big)
\\
&=O\Big(\sum_{J\in \cJ}n^{-2v(J)}q^{- ((\Delta+1)(v(J)-1)-1)/2}
+r^{-2}\sum_{J\in \cJ'}\sum_{S\in\cS'}n^{-2v(J)}q^{- ((\Delta+1)v(J)-1)/2}\Big)
\\
&=o\Big(q^{1/2}\sum_{J\in \cJ}n^{-2v(J)}n^{v(J)-1}
+q^{1/2}r^{-2}\sum_{J\in \cJ'}\sum_{S\in\cS'}n^{-2v(J)}n^{v(J)}\Big)
\\
&=o(q^{1/2}n^{-1}+q^{1/2}r^{-1})=o(r^{-1}\log^{-1}n).
\end{align*}
Therefore, as $\mu=\omega(r\log n)$, and $\frac{\delta}{\mu^2}=o(r^{-1}\log^{-1}n)$, by Lemma~\ref{lemma:Janson}, the probability that there is no graph in $\cH^+$ in $G'_h$ is at most $\exp(-\tfrac{\mu^2}{2(\mu+\delta)})=\exp(-\omega(r\log n))$, completing the proof of Lemma~\ref{lemma:CompleteEmbedding} in Case I.

\smallskip

Let us assume now we are in Case II, with some $S_0$ with no edges between $S_0\in \cS'$ and $F_{h-1}$ in $F_h$.
Label $V(S_0)=\{v_1,\ldots,v_s\}$. Let $\cH$ be a maximal set of copies of $S_0$ in the complete graph with vertex set $U$, where each copy $H$ of $S_0$ has a different vertex set, $\{v_{H,1},\ldots,v_{H,s}\}$ say, so that  $v_i\mapsto v_{H,i}$ is an embedding of $S_0$, and $v_{H,i}\in B'(v_{S_0,i})$ for each $1\leq i\leq s$.

Note that if we have some graph $H\in \cH$ in $G'_h$, then we are
done, as then $V(H)\in E(L_{S_0})$. From~\eqref{helpful}, we have $|\cH|=\Omega(n^s)$, so, with very
similar calculations to Case II in the proof of
Lemma~\ref{lemma:proofalmost}, we have that the probability that there exists no graph from $\cH$ in $G'_h$ is at most $\exp(-\omega(n))\le\exp(-\omega(r\log (n/r))$, as required.
\end{proof}


\section{Concluding remarks}\label{sec:conclude}

\subsection*{Extending Theorem~\ref{theorem:main} to smaller maximum degrees}
Theorem~\ref{theorem:main} can be easily  extended to $\Delta \le 3$ using basically the same approach as in Section~\ref{sec:max}.
The definition of the `dense spots', however, has to be slightly adapted to each case, but since it is straightforward,  we omit the details.
There is no extension to $\Delta=4$ of Theorem~\ref{Theorem:Ferber} due to
the existence of one problematic dense spot: a triangle attached to the
rest of the graph with two pendant edges at each vertex. This means that,
using a similarly defined set of subgraphs $\cF$ as in the proof of
Theorem~\ref{theorem:main}, we cannot show that one of these subgraphs
appears whp in $G(n,\omega(n^{-2/5}))$ (i.e. we cannot prove \ref{AA1} in Definition~\ref{suitabledefn}),  and this prevents our approach from extending to this case.

\subsection*{Using our method}
Our main technical theorem, Theorem~\ref{newtheorem}, provides a new
general purpose tool for finding spanning structures~$F$ in randomly
perturbed graphs $G_\alpha\cup \Gnp$.
To use Theorem~\ref{newtheorem}, it is sufficient to show that $F$ has
a collection of subgraphs which is $(\alpha,p)$-suitable.
Our approach avoids the regularity lemma, which appears in many
previous proofs for results concerning spanning structures in $G_\alpha\cup
\Gnp$~\cite{BTW_tilings,krivelevich2015cycles,krivelevich2015bounded}.
In particular, our approach provides simpler proofs for recent results concerning bounded degree
spanning trees and factors, as we sketch in the following.

\subsection*{Spanning trees}
Krivelevich, Kwan and Sudakov~\cite{krivelevich2015bounded} showed that,
for any $\alpha,\Delta>0$, if $p= \omega (1/n)$ and $T$ is an $n$-vertex
tree with maximum degree at most $\Delta$, then $G_\alpha\cup \Gnp$
contains a copy of $T$ whp. We can reprove this result using
Theorem~\ref{newtheorem} as follows.
Fixing $\alpha>0$ and $\Delta>0$, let $\eps=\eps(\alpha,\Delta)$ be as given in Definition~\ref{suitabledefn}. Let $p=\omega(1/n)$ and let $T$ be a tree with $n$ vertices and maximum degree at most $\Delta$. Clearly $T$ contains some subtree $T'$ with just over $(1-\eps)n$ vertices, pick such a subtree and let $\cF=\{T'\}$.
By the work of Alon, Krivelevich and Sudakov~\cite{alon2007embedding}, we
know that $G(n,p/2)$ whp contains a copy of $T'$, and therefore \ref{AA1} holds for $\cF$. Furthermore, \ref{AA2} easily holds without even recourse to the random edges in $G(2n,p/6)$. The copy of $T'$ can be extended by iteratively adding leaves. When we wish to add a leaf to a vertex $w$, say, to embed $v\in V(F)\setminus V(F')$, as $|B(v)\cap N_G(w)|\geq 4\eps n$, there will be many vertices to choose from in $B(v)\cap N_G(w)$ which are not yet in the embedding.
Thus, $\cF$ is $(\alpha,p)$-suitable and Theorem~\ref{newtheorem} applies.

\subsection*{Factors}
Balogh, Treglown and Wagner~\cite{BTW_tilings} showed that for every~$H$,
if $p=\omega(n^{-1/m_1(H)})$, then $G_\alpha\cup \Gnp$ contains an
$H$-factor whp. Again, we can use Theorem~\ref{newtheorem} to easily
reprove this result.
Indeed, let $F$ be an $H$-factor and $\cF$ be the set of subgraphs of $F$
consisting of disjoint copies of $H$ which cover at least $(1-\eps)n$
vertices. By Theorem~\ref{thm:factors} we have that~\ref{AA1} holds
for~$\cF$.
Another simple application of Janson's inequality gives that~\ref{AA2}
holds as well.

\subsection*{Randomly perturbed hypergraphs}
Recently generalisations of the model of randomly perturbed graphs to
hypergraphs attracted much attention. Again, the union of a binomial random $r$-uniform
hypergraph $G^{(r)}(n,p)$ and a deterministic $r$-uniform hypergraph $G_\alpha$ satisfying a certain minimum
degree condition is considered.
In the hypergraph setting several different notions of minimum degree are possible.

The study of randomly perturbed hypergraphs was initiated by
Krivelevich, Kwan and Sudakov ~\cite{krivelevich2015cycles} who considered
hypergraphs $G_\alpha$ with \emph{collective minimum degree} $\alpha n$, that is,
each $(r-1)$-set of vertices of $G_\alpha$ is contained in at least $\alpha n$ edges.
A \emph{loose Hamilton cycle} in an $r$-uniform hypergraph on $n=(r-1)k$
vertices for some integer~$k$, is a labelling
of its vertices by $0,\dots,n-1$ such that $\{i,\dots,i+(r-1)\}$ is an
edge for each $i=(r-1)j$ with $j\in\{0,\dots,k-1\}$, where indices are taken modulo~$n$. In other
words, consecutive edges of a loose Hamilton cycle overlap in exactly one
vertex.
We remark that, for loose Hamilton cycles, a Dirac-type theorem is known~\cite{keevash2011loose}.
Krivelevich, Kwan and Sudakov ~\cite{krivelevich2015cycles} proved that,
for any $G_\alpha$ with collective minimum degree $\alpha n$, the addition
of random edges with edge probability $c(\alpha)n^{-r+1}$  (where
$c(\alpha)>0$ depends on $\alpha$ only) is sufficient to create whp perfect
matchings as well as loose Hamilton cycles.
Comparing this to the threshold for matchings and loose Hamilton cycles in random hypergraphs, which is $n^{-r+1} \log n$ \cite{dudek2011loose,Fri10,JohanssonKahnVu_FactorsInRandomGraphs}), this again differs by a factor of $\log n$.

%
%
%
%

Different minimum degree conditions were considered by
McDowell and Mycroft~\cite{MM_HyperPeturbed}. An $r$-uniform hypergraph $G_\alpha$ has \emph{minimum
  $\ell$-degree} at least $\alpha n^{r-\ell}$ if
each $\ell$-set of vertices of $G_\alpha$ is contained in at least $\alpha n^{r-\ell}$ edges.
An \emph{$\ell$-overlapping} cycle is defined analogously to a loose
Hamilton cycle, but with consecutive edges overlapping in exactly~$\ell$
vertices. A \emph{tight Hamilton cycle} in an $r$-uniform hypergraph is an
$(r-1)$-overlapping Hamilton cycle.
McDowell and Mycroft~\cite{MM_HyperPeturbed} showed that for
$\ell$-overlapping Hamilton cycles with $\ell \ge 3$ it is possible to save a
polynomial factor $n^\varepsilon$ on the edge probability in randomly perturbed $r$-uniform
hypergraphs $G_\alpha\cup G^{(r)}(n,p)$ compared to $G^{(r)}(n,p)$ alone,
under the assumption that $G_\alpha$ has minimum $\ell$-degree at least
$\alpha n^{\ell}$ and minimum $(r-\ell)$-degree at least $\alpha n^{r-\ell}$.
This result was extended by Bedenknecht, Han, Kohayakawa, and
Mota~\cite{BHKM_powers} to powers of tight Hamilton cycles, with the
additional assumption of collective minimum degree at least $\alpha n$ with $\alpha>c_{r,\ell}$.

The weaker notion of minimum $1$-degree was studied in the context of
randomly perturbed hypergraphs by Han and Zhao~\cite{HanZhao2017}.  It is
not difficult to see that an $r$-uniform hypergraph with minimum collective
degree at least $\alpha n$ has minimum $1$-degree at least
$\alpha \binom{n-1}{r-1}$.  Hence, Han and Zhao~\cite{HanZhao2017}
strengthen the results of Krivelevich, Kwan and Sudakov by proving that
adding $c(\alpha)n$ random edges to $G_\alpha$, whp creates a perfect
matching and a loose Hamilton cycle. Furthermore, adding $c(\alpha)n^{r-1}$
random edges to $G_\alpha$ gives rise to a tight Hamilton cycle. Both these
results, as well as those from~\cite{krivelevich2015cycles}, use the
regularity method.

The absorption technique we introduce in this paper can be
extended to the randomly perturbed hypergraph model, and may allow some
progress. In particular, we have confirmed that an easy extension of our
method gives the appearance threshold for a perfect matching and a loose
Hamilton cycle in this model, recovering the results
of~\cite{krivelevich2015cycles} and~\cite{HanZhao2017}.

We note that the third and fourth of the current
authors~\cite{parczyk2016spanning} have extended the result of
Riordan~\cite{Riordan} to hypergraphs. Similar extensions of
Theorems~\ref{Theorem:Ferber} and~\ref{theorem:main} however remain open
and would be very interesting.

\subsection*{Universality}
We believe that a universality result corresponding to our main theorem
holds as well. That is, we believe that when $p=\omega(n^{-\frac{2}{\Delta+1}})$ the
randomly perturbed graph
 $G_\alpha\cup \Gnp$ contains whp a copy of every graph in
 $\cF(n,\Delta)$ simultaneously.
However, our use of Riordan's result~\cite{Riordan}, which was proved by
second moment calculations, makes it unlikely that our techniques can be
used to obtain such a result.
Thus, new ideas are required.
Similarly, $p_\Delta$ is commonly believed to be the threshold for $\Gnp$
to contain a copy of every graph in $\cF(n,\Delta)$ simultaneously, but the current
methods to attack this problem (see the discussion after
Theorem~\ref{Theorem:Ferber}) require an edge probability in distinct
excess of this conjectured threshold.

In the case of spanning bounded degree trees, in joint work with Han and
Kohayakawa we establish the following universality result
in~\cite{BHKMPP17}. We show that $G_\alpha\cup G(n,c(\alpha,\Delta)/n)$
simultaneously contains all spanning trees of maximum degree at most
$\Delta$.

\section*{Acknowledgement}

We would like to thank the referee for their valuable comments.

\bibliographystyle{amsplain_yk}
\bibliography{literatur}

\end{document}